%% file: feast_discrete.tex
\documentclass[12 pt]{amsart}

\input{definition}

\graphicspath{{figs/}}

\begin{document}

\title[Discretization errors in subspace iteration]{Spectral discretization errors in filtered subspace iteration}

\author[J. Gopalakrishnan]{Jay Gopalakrishnan}
\address{Portland State University, PO Box 751, Portland, OR 97207-0751, USA}
\email{gjay@pdx.edu}

\author[L.~Grubi\v{s}i\'{c}]{Luka Grubi\v{s}i\'{c}}
\address{University of Zagreb, Bijeni\v{c}ka 30, 10000 Zagreb, Croatia}
\email{luka.grubisic@math.hr}

\author[J. Ovall]{Jeffrey Ovall}
\address{Portland State University, PO Box 751, Portland, OR 97207-0751, USA}
\email{jovall@pdx.edu}

\date{\today}

\begin{abstract}
  We consider filtered subspace iteration for approximating a cluster of
  eigenvalues (and its associated eigenspace) of a (possibly
  unbounded) selfadjoint operator in a Hilbert space.  The algorithm is motivated
  by a quadrature approximation of an operator-valued contour integral
  of the resolvent. Resolvents on infinite dimensional spaces are
  discretized in computable finite-dimensional spaces before the
  algorithm is applied.  This study focuses on how such
  discretizations result in errors in the eigenspace approximations
  computed by the algorithm. The computed eigenspace is then used to
  obtain approximations of the eigenvalue cluster.  Bounds for the
  Hausdorff distance between the computed and exact eigenvalue
  clusters are obtained in terms of the discretization parameters
  within an abstract framework.  A realization of the proposed
  approach for a model second-order elliptic operator using a standard
  finite element discretization of the resolvent is described. Some
  numerical experiments are conducted to gauge the sharpness of 
  the theoretical estimates.
\end{abstract}

\thanks{This work was partially supported by the AFOSR
  (through AFRL Cooperative Agreement \#18RDCOR018 and grant
  FA9451-18-2-0031), the Croatian Science Foundation grant HRZZ-9345, 
  bilateral Croatian-USA grant (administered jointly by Croatian-MZO and NSF) and 
  NSF grant DMS-1414365. The numerical studies
  were facilitated by the equipment acquired using NSF's Major
  Research Instrumentation grant DMS-1624776.}

\maketitle

\section{Introduction}  \label{sec:introduction}

The goal of this study is to provide an analysis of discretization
errors when a subspace iteration algorithm is employed to compute
eigenvalues of selfadjoint partial differential operators. Instead of
a specific differential operator, we consider a general linear,
closed, selfadjoint operator $\cA: \dom(\cA) \subseteq \cH \to\cH$
(not necessarily bounded) in a complex Hilbert space $\cH$, whose
(real) spectrum is denoted by $\Sigma(\cA)$.  We are interested in
computationally approximating a subset $\Lambda$ of the spectrum that
consists of a finite collection of eigenvalues of finite multiplicity.
{\em Filtered subspace iteration}, a natural generalization of the
power method, for approximating $\Lambda$ and its corresponding
eigenspace (invariant subspace) is roughly described as follows.
First, the eigenspace of the cluster $\Lambda$ is transformed to the
dominant eigenspace of another, bounded operator called the
``filter.''  Next, a subspace iteration is applied using the bounded
filter.  Starting with an initial subspace (usually chosen randomly),
the bounded operator is repeatedly applied to it, generating a
sequence of subspaces that {approximates} the eigenspace of $\Lambda$.
Approximations of $\Lambda$ are obtained from the approximate
eigenspaces by a Rayleigh-Ritz procedure.  To apply this filtered
subspace iteration in practice requires finite rank approximations of
the resolvent at a few points along a contour, obtained by some discretization process, and it is the
errors incurred by such discretizations that are investigated here.

The exact eigenspace, namely the span of all the eigenvectors associated
with elements of $\Lambda$, is denoted by $E$. Then $ m = \dim E,$
being the sum of multiplicities of each element of $\Lambda$, is
finite, and we assume $m \ge 1.$ Throughout this paper, the
multiplicity $\ell$ of an eigenvalue $\lambda$ of an operator refers
to its algebraic multiplicity, i.e., $\lambda$ is a pole of order
$\ell$ of the resolvent of that operator. Recall that, for a
selfadjoint operator $A$, the algebraic multiplicity of $\lambda$
coincides with its geometric multiplicity, $\dim\ker(\lambda-A)$.

As mentioned above, the idea behind filtered subspace iteration is to
transform $E$ into the dominant eigenspace of certain filter
operators.  We shall see in the next section that the construction of
these filters can be motivated by approximations of a Dunford-Taylor
contour integral.  There has been a resurgence of interest in contour
integral methods for eigenvalues due to their excellent
parallelizability~\cite{Austin2015, Beyn2012, HuangStrutSun16,
  SakurSugiu03,GuttePolizTang15}.  Following~\cite{Austin2015}, we
identify two different classes of methods in the existing literature
that use contour integrals for computation of a targeted cluster of
{matrix eigenvalues}.  One class of methods, that often goes by the
name SSM~\cite{SakurSugiu03} (see also~\cite{Beyn2012,Sakurai}),
approximates $\Lambda$ by the eigenvalues of a system of moment
matrices based on contour integrals. The moment matrices are obtained
by approximating the integrals by a quadrature, and the spectral
approximation error depends on the accuracy of the quadrature.

The other class of methods, more related to our present contribution,
in that they apply filtered subspace iterations to matrices, are
referred to by the name FEAST~\cite{Polizzi2009} (see
also~\cite{GuttePolizTang15,Polizzi2014}).  These algorithms also use
quadratures to approximate a contour integral, but the eigenvalues are
approximated by a Rayleigh-Ritz procedure using the original
matrix. In our view, the use of quadratures in such algorithms is
essentially different from their use in SSM-like approaches.
Quadratures in FEAST are only used to develop the filter used in a
subspace iteration. A consequence of this is that the quadrature error
is not as relevant in FEAST as in SSM. The analysis in this paper will
show this in precise terms.

When $A$ is a differential operator on an infinite-dimensional space,
some approximations to bring the computations into finite-dimensional
spaces are necessary. The central concern in this paper is the study
of how these approximations affect the final spectral approximations
that filtered subspace iteration yields asymptotically.  The main
technical difficulty in analyzing discretization errors for the
unbounded operator eigenproblem is that many of the existing standard
tools~\cite{BabusOsbor91} (applicable for compact operators) are
\revv{not directly applicable to our situation. We present an abstract
  framework that allows one to study approximation of spectral
  clusters of unbounded selfadjoint operators with compact resolvent.
  Very general discretizations are allowed through a set of abstract
  assumptions.}

To quickly outline the approximation approach studied in this paper,
recall that the spectral projector onto $E$, which we denote by $S$,
is characterized by a Dunford-Taylor contour integral of the
resolvent~$R(z)$. Its $N$-term quadrature approximation is denoted by
$S_N$. In the expression defining $S_N$, when $R(z)$ is replaced by a
computable finite-rank approximation $R_h(z)$, we obtain $S_N^h$, a
practically computable filter.  Here $h$ is some discretization
parameter (such as the grid spacing) inversely related to a
computational finite-dimensional space.  By repeated application of
$S_N^h$, the iteration produces a sequence of subspaces
$\{\smash{\Eh\ell} : \ell=1,2, \ldots\}$, which we study.  Due to the
approximations, it is not obvious if the iterates $\smash{\Eh\ell}$
converge to some limit~$E_h$.  We begin our analysis by showing that
they do under certain sufficient conditions, after which we focus on
analyzing the limit.

To summarize the novelty of this work, this is the first work to study
the effect of the {\em discretization parameter~$h$} (in addition
to~$N$).  The errors in eigenspace approximations often need to be
measured in stronger norms than the base $\cH$-norm. The example of
elliptic differential operators on $\cH = L^2(\om)$ illustrates the
need to measure eigenfunction errors in a stronger norm like the
$H^1(\om)$-norm. To our knowledge, this is the first work to give
bounds for eigenspace discretization errors in $\cH$-norm as well as a
stronger $\cV$-norm.  We provide the first result showing that the
Hausdorff distance between the eigenvalue cluster computed by filtered
subspace iteration and $\Lambda$ converges to zero at predictable
rates as the discretization parameter $h\to 0$.  To highlight one more
conclusion from our analysis, increasing $N$ has little effect on the
spectral discretization error as measured by the gap between $E$ and
$E_h$ (although it may affect \cite{feast_infinite} the speed of
convergence of the filtered subspace iteration to $E_h$).

The rest of the paper is organized as follows.  In
Section~\ref{sec:preliminaries}, we describe precisely the
above-mentioned process of double approximation (going from $S$ to
$S_N^h$) and introduce the necessary assumptions for the error
analysis. Section~\ref{sec:limit} introduces the space to which 
filtered subspace iteration using $S_N^h$ converges. Bounds for the gap between
computed and exact eigenspaces are proved in Section~\ref{sec:space}.
Eigenvalue errors are then bounded using the square of this gap.
Analysis of a standard finite element discretization of the resolvent
of a model operator in Section~\ref{sec:fem} provides an example of
how abstract conditions on the resolvent might be verified in
practice. The practical performance of the algorithm with the
discretization is reported in Section~\ref{sec:Numerics}.


\section{Preliminaries}  \label{sec:preliminaries}

Let $\cA$, $\Lambda$ and $E$ be as discussed previously.  As already
mentioned, filters are linear operators on $\cH$ having $E$ as their
dominant eigenspace, in the sense made precise below.  

Suppose that $\Gamma\subset \C\setminus\Sigma(A)$ is a positively
oriented, simple, closed contour that encloses $\Lambda$ and excludes
$\Sigma(A)\setminus\Lambda$, and let $G\subset\C$ be the open set
whose boundary is $\Gamma$.  By the Cauchy Integral Formula,
\begin{align}\label{character}
r(\xi)=\frac{1}{2\pi\ii}\oint_\Gamma(z-\xi)^{-1}\,dz=
\begin{cases}
1,& \,\xi\in G,
\\
0,&\xi\in \mathbb{C}\setminus(G\cup\Gamma).
\end{cases}
\end{align}
Thus $r(\xi)$ equals a.e.\ the indicator function of $G$ in $\C$.
The associated (orthogonal) spectral projection $S:\cH\to\cH$ is the
bounded linear operator given by the Dunford-Taylor integral 
\begin{align}\label{SpectralProjection}
S=\frac{1}{2\pi\ii}\oint_\Gamma R(z)\,dz,
\end{align}
where $R(z) = (z-\cA)^{-1}$ is the resolvent, a bounded linear
operator on $\cH$ for each $z \in \Gamma$.  Since $\Gamma$ encloses
$\Lambda$ and no other element of $\Sigma(A)$, its well known that
\begin{align}\label{InvariantSubace}
E=\ran(S).
\end{align}
Furthermore, by functional calculus
(see \cite[Theorem~VIII.5]{ReedSimon72},~\cite[Theorem 5.9]{Schmu12}
or \cite[Section~6.4]{BuhlerSalamon18}),
if $(\lambda,\phi) \in \Sigma(\cA)\times\dom(\cA)$ satisfies
$A \phi = \lambda \phi$,
then $S\phi = r(A) \phi = r(\lambda)\phi$.
Since $r(\lambda)$ equals $1$ for all $\lambda \in \Lambda$
and equals $0$ for all other elements of $\Sigma(A)$,
the desired eigenspace $E$ of $A$ 
is now the dominant eigenspace of $S = r(A)$.
In this sense, $S$ is an {\em ideal filter.}


Motivated by quadrature approximations of~\eqref{character}, in the
same spirit as~\cite{Beyn2012,GuettelD,
  Polizzi2009,Polizzi2014,SakurSugiu03}, 
we approximate $r(\xi)$ by 
\begin{align}\label{ContourQuadrature}
r_N(\xi)=w_N + \sum_{k=0}^{N-1} w_k(z_k-\xi)^{-1}~,
\end{align}
for some  $w_k, z_k\in\C.$ 
The corresponding {\em rational filter} is the operator
\begin{align}\label{ApproxSpectralProjection}
S_N&=r_N(\cA)=w_N + \sum_{k=0}^{N-1}w_k R(z_k),
\end{align}
which can be viewed as an approximation of~$S$.  It is common to refer
to $S_N$, as well as the rational function $r_N(\xi)$, as the {\em
  filter.}  As in the case of~$S$, if
$(\lambda,\phi) \in \Sigma(\cA)\times\dom(\cA)$ satisfies
$A \phi = \lambda \phi$, then $S_N\phi = r_N(\lambda)\phi$.  In
particular, the set $\Lambda$ of eigenvalues of interest have been
mapped to $\{ r_N(\lambda): \lambda \in \Lambda \}$ by the filter.

These mapped eigenvalues are dominant eigenvalues of $S_N$ if
\begin{align}\label{FilterRequirement}
  \min_{\lambda\in\Lambda}|r_N(\lambda)|
  > \sup_{\mu\in \Sigma(\cA)\setminus\Lambda}|r_N(\mu)|
\end{align}
holds. This dominance can be obtained provided $\Lambda$ is strictly
separated from the remainder of the spectrum. To quantify the
separation, we consider the following strictly separated subsets of
$\RRR$ centered around $y\in \RRR$
\[
I_\gamma^y = \{ x \in \R: | x- y| \le \gamma\}, \qquad 
O_{\delta,\gamma}^y = \{ x \in \R: |x - y| \ge (1+\delta)\gamma\}.
\]
for some positive numbers $\gamma$ and $\delta$.  If the spectral
cluster of interest is within $I_\gamma^y$, then the number $\delta$
provides a measure of the relative gap between it and the rest of the
spectrum---relative to the radius $\gamma$ of the interval wherein we
seek eigenvalues.  Using the numbers $y, \gamma,$ and $\delta$, define
\begin{align}\label{ContractionFactor2}
W=\sum_{k=0}^{N}|w_k|,\quad\quad
\hat\kappa=
\frac{  \displaystyle\sup_{x \in O_{\delta,\gamma}^y}  |r_N(x)|}
{ \displaystyle \inf_{x \in I_\gamma^y}|r_N(x)|}.
\end{align}
These definitions help us to formulate the following assumption on the
filter and cluster separation.

\begin{assumption}
  \label{asm:rN}
  We assume there are $y\in\RR$,
  $\delta>0$ and $\gamma>0$ such that
  \begin{align}\label{SpectralGap}
    \Lambda\subset I_\gamma^y,
    \qquad 
    \Sigma(\cA)\setminus\Lambda \subset
    O_{\delta,\gamma}^y.
  \end{align}
  We assume that $r_N$ is a rational function of the
  form~\eqref{ContourQuadrature} with the property that
  $z_k \notin\overline{\Sigma(\cA)},$ $W<\infty$, and $\hat \kappa <1$.
\end{assumption}

Note that if $\hat\kappa < 1$, then \eqref{FilterRequirement} holds.
When an $N$-point trapezoidal rule is used for quadrature
approximation we obtain an $r_N$ as in~\eqref{ContourQuadrature} with
$w_N=0$. When the Zolotarev rational approximation of $r(\xi)$ is used
to construct $r_N$, the term $w_N$ is nonzero~\cite{GuttePolizTang15}.
In Section~\ref{sec:Numerics}, we shall use the so-called {\em
  Butterworth filter} which is obtained by setting $w_N=0$ and for
$k=0, \ldots, N-1$,
\begin{align}\label{CircleQuad}
z_k=\gamma e^{\ii
  (\theta_k+\phi)}+y,\quad\quad w_k=\gamma
e^{\ii (\theta_k+\phi)}/N.
\end{align}
with $\theta_k=2\pi k/N$ and $ \phi=\pm\pi/N.$ This filter as well as
several other examples of $r_N$ satisfying Assumption~\ref{asm:rN} are
studied in~\cite{feast_infinite}.

Next, we introduce a subspace $\cV \subseteq \cH$, motivated by the
need to prove results that bound discretization errors in norms
stronger than the $\cH$-norm. We place the following assumption and
give example classes of operators where the assumption
holds.
\begin{assumption}
  \label{asm:Vshort}
  Suppose there is a Hilbert space $\cV \subseteq \cH$ such that
  $E \subseteq \cV$, there is a $C_\cV>0 $ such that for all
  $u \in \cV$, $ \| u \|_\cH \le C_\cV \| u \|_\cV$, and $\cV$ is an
  invariant subspace of $R(z)$.
\end{assumption}

\begin{example}[$\cV$ is the whole space]
  \label{eg:case:A}
  Set $\cV = \cH$, with $(\cdot,\cdot)_\cV=(\cdot,\cdot)_\cH$. In this
  case it is obvious that all statements of Assumption~\ref{asm:Vshort}
  hold.~\hfill$\Box$
\end{example}


\begin{example}[$\cV$ is the domain of a positive form]
  \label{eg:case:B}
  Suppose $a(u,v)$ is a densely defined closed sesquilinear 
  Hermitian form { on $\cH$} and there is a $\delta>0$ such that 
  \begin{equation}
    \label{eq:3}
   a(v,v) \ge \delta \| v \|_{\cH}^2, \qquad v \in \dom(a). 
  \end{equation}
  Set 
  \[
  \cV = \dom(a), 
  \qquad
  \| v \|_\cV = a(v,v)^{1/2}.
  \]
  To show that Assumption~\ref{asm:Vshort} holds in this case, first
  set the operator $A$ to be the closed selfadjoint {\em operator
    associated with the form}, namely it satisfies $a(u,v) = (Au,v)$
  for all $u \in \dom(A) \subseteq \dom(a)$ and all $v \in \dom(a)$
  (see the first representation theorem \cite[TheoremVI.2.1]{Kato1995}
  or \cite[Theorem~10.7]{Schmu12}).  Note that, in this case, $A$ is a
  positive operator.  Hence $A$ has a unique selfadjoint positive
  square root \cite[Theorem~V.3.35]{Kato1995}, denoted by $A^{1/2}$,
  that commutes with any bounded operator that commutes with $A$.  By
  the second representation theorem~\cite[Theorem VI.2.23]{Kato1995},
  the form domain is characterized by $\dom(a) = \dom(A^{1/2})$, and
  $\|v\|_\cV=\|A^{1/2}v\|_\cH$ for $v\in\cV$.  The strict positivity
  of $a$ ensures that both $A$ and $A^{1/2}$ are invertible on their
  respective domains.  

  Since $a$ is closed, $\cV$ is complete.  Due to~\eqref{eq:3}, $\cV$
  is continuously embedded in $\cH$, with the constant
  $C_\cV = \delta^{-1/2}$.  The exact eigenspace $E$ is contained in
  $\dom(A) \subseteq \dom(A^{1/2}) = \cV$.  Since $A^{1/2}$ and
  $A^{-1/2}$ commutes with $R(z)$, for any $v, w \in \cV$, we have for
  any $v\in\cV=\dom(A^{1/2})$ and $z$ in the resolvent set of $A$,
\begin{align*}
  R(z)v&=(z-A)^{-1}v=A^{-1/2}(z-A)^{-1}A^{1/2}v~.
\end{align*}
Since $\ran(A^{-1/2})=\dom(A^{1/2})=\cV$, we see that
$R(z)\cV\subseteq\cV$.  Thus
Assumption~\ref{asm:Vshort} is verified.~\hfill$\Box$
\end{example}

\begin{example}[$\cV$ is a graph space]
    \label{eg:case:C}
    Given $A$, put $\cV=\dom (A) \subseteq \cH$ and endow the set $\cV$
    with the topology of the graph norm
    \[
    \| v \|_{\cV} = 
    \left( 
      \| v \|_{\cH}^2 + \| \cA v \|_{\cH}^2
    \right)^{1/2}, \qquad  v \in \cV.
    \]
    We claim that Assumption~\ref{asm:Vshort} holds in this case.
    Indeed, since $A$ is closed, the graph norm makes $\cV$ into a
    Hilbert space.  Obviously $E\subset \cV$ and $\cV$ is continuously
    embedded into $\cH$ with $C_\cV=1$. Since $A$ commutes with $R(z)$
    for any $z$ in the resolvent set of $A$, we have
    $ R(z) \dom(A) \subseteq \dom(A)$.~\hfill$\Box$
\end{example}

The next essential ingredient in our study is the approximation of
$R(z)$.  When $\cA$ is a differential operator on an
infinite-dimensional space, to obtain numerical spectral
approximations, we perform a discretization to approximate the
resolvent of $\cA$ in a computable finite-dimensional
space. Accordingly, let $\cV_h$ be a finite-dimensional subspace of
$\cV$, where $h$ is a parameter inversely related to the finite
dimension, e.g., a mesh size parameter $h$ that goes to $0$ as the
dimension increases. Let $R_h(z):\cH \to\cV_h$ be a finite-rank
approximation to the resolvent $R(z)$ satisfying the following
assumption.
\begin{assumption}
  \label{asm:Rlim}
  Assume that the operators $R_h(z_k)$ and $R(z_k)$ are bounded in
  $\cV$ and satisfy
  \begin{equation}
    \label{eq:Rh-R}
    \lim_{h \to 0} \| R_h(z_k) - R(z_k) \|_{\cV} = 0
  \end{equation}  
  for all $k=0, 1,\ldots, N-1$.
\end{assumption} 

{\revv{Note that this assumption implies that $R(z_k)$, being the
    limit of finite-rank operators, is compact in $\cV$. Its also
    compact as an operator on $\cH$ due to
    Assumption~\ref{asm:Vshort}.  Consequently $R(z)$ is compact for
    all $z$ in the resolvent set. Relaxing Assumption~\ref{asm:Rlim}
    to go beyond operators with compact resolvent is outside the scope
    of the current work.}}

Consider the approximation of $S_N$ given by
\begin{align}\label{eq:SNh}
S^h_N&=w_N + \sum_{k=0}^{N-1}w_k R_h(z_k).
\end{align}
In view of Assumption~\ref{asm:Rlim}, we shall from now on view both
$S_N$ and $S_N^h$ as bounded operators on~$\cV$. Note that
\revv{$S_N^h$ need not be  selfadjoint}. In
Section~\ref{sec:fem}, we shall consider an example of $S_N^h,$
obtained by a standard finite element discretization of $R(z)$
\revv{based on symmetrically located $z_k$, that
is selfadjoint. But in general $S_N^h$ may fail to be selfadjoint
due to the configuration of $\{z_k\}$ or due to
the properties of the discretization (see e.g.,~\cite{feast_dpg}).}

With the resolvent discretization, filtered subspace iteration can be
described mathematically in very simple terms. Namely, starting with a
subspace $\Eh 0 \subseteq \cV_h$, compute
\begin{equation}
  \label{eq:1}
  \Eh\ell = S^h_N \Eh {\ell -1}, 
  \qquad 
  \text{ for } 
  \ell=1,2,\ldots. 
\end{equation}
Of course, in practice, one must include (implicit or explicit)
normalization steps and maintain a basis for the spaces
$\smash{\Eh \ell}$, but these details are immaterial in our ensuing
analysis.  The convergence of the FEAST algorithm in Euclidean
($\ell^2$ and matrix-based) norms was previously studied
in~\cite{GuttePolizTang15,Saad16}. In Section~\ref{sec:limit}, we
shall utilize some of their ideas to show that~\eqref{eq:1} converges in
$\cV$, despite the perturbations caused by the above-mentioned
resolvent approximations. Further studies on iterative speed of
convergence of~\eqref{eq:1} may be found in~\cite{feast_infinite}.
Here however, we are solely interested in studying the {\em
  discretization errors found in the final asymptotic product of the algorithm},
i.e., the discretization errors in what the algorithm outputs as the
``limit space'' when~\eqref{eq:1} converges.


\section{The limit space}
\label{sec:limit}

The purpose of this section is to identify to what space convergence
of~\eqref{eq:1} might happen.  We also briefly examine in what
sense $\smash{\Eh\ell}$ converges to it.

In view of Assumption~\ref{asm:Vshort}, $\cV$ is an invariant subspace
of the resolvent. Hence in the remainder of the paper, we will proceed
{\em viewing $S_N$ and $S_N^h$ as operators on $\cV$}.  To measure the
distance between two linear subspaces $M$ and $L$ of $\cV$, we use the
standard notion of gap \cite{Kato1995} defined by
\begin{equation}
  \label{eq:gapV}
\gap_\cV( M, L ) = \max \left[ \sup_{m \in U_M^\cV} \dist{\cV} ( m, L), \;
  \sup_{l \in U_L^\cV} \dist{\cV} ( l, M) \right].
\end{equation}
Here and throughout, for any linear subspace $M\subseteq \cV$, we use
$U_M^\cV$ to denote its unit ball $ \{ w \in M: \; \| w\|_\cV = 1\}$.

Recall that $E = \ran S$, the exact eigenspace corresponding to
eigenvalues $\lambda_1, \ldots, \lambda_m$ of $\cA$ that we wish to
approximate. If Assumption~\ref{asm:rN} holds, then 
the  operator $S_N = r_N(A)$ 
has  dominant eigenvalues 
\[
\mu_i = r_N(\lambda_i), \qquad i=1,2, \ldots, m,
\]
strictly separated in absolute value from the remainder of
$\Sigma(S_N)$.  In particular, since $\hat \kappa <1$, we have 
$\mu_i \ne 0$ for $i\leq m$, and letting $\mu_* =
\sup\{|\mu|: \mu \in \Sigma(S_N) \setminus \{\mu_1, \ldots \mu_m\}\}$,
\begin{equation}
  \label{eq:11}
  \mu_* < |\mu_i|, \quad i=1,2,\ldots, m.
\end{equation}
In view of these facts, we can find a simple rectifiable curve $\Theta$ in
the complex plane that encloses $\{ \mu_1, \ldots, \mu_m\}$ and lies
strictly outside the circle of radius $\mu_*$. 
In particular, $\Theta$ encloses no other element of
$\Sigma(S_N)$.  Define the spectral projector of $S_N$ by
\[
P_N = \frac{1}{ 2 \pi \ii} \oint_{\Theta} ( z - S_N)^{-1} \, dz.
\]
Then $E_N = \ran P_N$ is the eigenspace of $S_N$ corresponding to its
eigenvalues $ \mu_1, \ldots, \mu_m$.

\begin{lemma}
  \label{lem:E=EN}  
  We have 
  $E_N=E$ and $P_N = S$. 
\end{lemma}
\begin{proof}
  Since $\dim E_N = \dim E = m$, it suffices to prove that
  $E \subseteq E_N$. If $e_i\in E$ is an eigenfunction of $A$ corresponding
  to the eigenvalue $\lambda_i\in\Lambda$, then
  $S_Ne_i=r_N(\lambda_i)e_i$, so $e_i\in E_N$.  Since $P_N$ and $S$
  are both orthogonal projectors and have the same range, they are the
  same operator.
\end{proof}

Next, observe that when  Assumption~\ref{asm:Rlim} is used after 
  subtracting the expression
for $S_N^h$ in~\eqref{eq:SNh} from that of $S_N$, we obtain 
\begin{equation}
  \label{eq:Shlim-V}
  \begin{aligned}
  \| S_N - S^h_N  \|_\cV
  & \le W
  \max_{k =0, \ldots, N-1} \| R_h(z_k) - R(z_k) \|_{\cV} \to 0
  \end{aligned}
\end{equation}
as $h$ goes to 0. Let us recall the standard ramifications of the
convergence of operators in norm given by~\eqref{eq:Shlim-V} (see
e.g.,~\cite[Theorem~IV.3.16]{Kato1995} or
\cite{AnselPalme68}). Namely, given an open disc enclosing an isolated
eigenvalue of $S_N$ of multiplicity~$\ell$, \eqref{eq:Shlim-V} implies
that for sufficiently small $h$, there are exactly $\ell$ eigenvalues
(counting multiplicities) of $S^h_N$ in the same disc. In particular,
this implies that, for sufficiently small $h$, the contour $\Theta$ is
in the resolvent set of $S^h_N$ and encloses exactly $m$ eigenvalues
(counting multiplicities) of~$S^h_N$, which we shall enumerate as
$ \mu_1^h, \mu_2^h, \ldots, \mu_m^h$. 
Hence, the integral
\[
P_h = \frac{1}{ 2 \pi \ii} 
\oint_\Theta \left(z - S^h_N\right)^{-1} \,dz
\]
is well defined. 

\begin{definition}  \label{def:Eh}
Let $E_h = \ran P_h$.    
\end{definition}

Clearly, $P_h$ is the spectral
projector of $S^h_N$ corresponding to the eigenvalues
$\mu_1^h, \mu_2^h, \ldots, \mu_m^h$. Hence,
\begin{equation}
  \label{eq:dimEh}
\dim E_h = m.  
\end{equation}
Note also that by construction of  $\Theta,$
\begin{equation}
  \label{eq:2}
  \mu^h_i \ne 0, \qquad i=1,2, \ldots, m.
\end{equation}
We now show that the above-defined $E_h$ is the limit space of subspace
iterates $E_h^{(\ell)}$.

\begin{theorem}
  \label{thm:iteratesgap}
  Starting with a subspace $\Eh 0 \subseteq \cV_h $ satisfying
  $\dim(\Eh 0)=\dim( P_h \Eh 0 ) =m$, we compute $\Eh\ell$ by~\eqref{eq:1}.
  Suppose Assumptions~\ref{asm:rN}--\ref{asm:Rlim} hold.  Then there
  is an $h_0>0$ such that, for all $h < h_0$,
  \[
  \lim_{\ell \to \infty} \gap_\cV( \Eh \ell, E_h) =0.
  \]
\end{theorem}
\begin{proof}
  {\em Step~1:} Recall from~\eqref{eq:dimEh} that $\dim(E_h) = m$ for
  sufficiently small~$h$.  Together with $P_h \Eh 0 \subseteq E_h$ and
  the assumption $\dim(P_h \Eh 0)=m$, this leads to the equality
  \begin{equation}
    \label{eq:PhEh=Eh}
    P_h \Eh 0 = E_h.    
  \end{equation}
  Thus
  \[
    P_h \Eh \ell = P_h (S^h_N)^\ell \Eh 0 = (S^h_N)^\ell P_h \Eh 0 =
    (S^h_N)^\ell E_h = E_h.
  \]
  In particular, this implies that
  $\dim( \Eh \ell) \ge \dim( P_h\Eh \ell) = \dim(E_h)$. Hence,
  \begin{equation}
    \label{eq:dimEhl}
    \dim( \Eh \ell) = \dim(E_h) = m, \qquad  \ell =0,1,2, \ldots.
  \end{equation}

  {\em Step~2:} Let $v_i$ be an eigenvector of $S_N^h$ corresponding
  to eigenvalue $\mu^h_i$.  We shall now find an approximant of $v_i$
  in $E_h^{(\ell)}$.  Due to~\eqref{eq:PhEh=Eh}, there is a $\qii 0$
  in $ \Eh 0$ such that $P_h \qii 0 = v_i$. Set
  \[
    \qii \ell = \left(\frac{1}{\mu^h_i}\right)^{\ell} (S_N^h)^\ell\, \qii 0.
  \]
  Clearly $\qii \ell$ is well defined due to~\eqref{eq:2} and is in
  $\Eh \ell$. Moreover,
  \begin{align*}
    v_i - \qii \ell 
    & = v_i - \mu_i^{-\ell} (S_N^h)^\ell \big[ P_h \qii 0 
      + (I - P_h) \qii 0\big]
      \\
    & =  - \mu_i^{-\ell} (S_N^h)^\ell (I - P_h) \qii 0,
  \end{align*}
  Since  $(I-P_h) \qii 0 = (I-P_h)^2 \qii 0 = (I-P_h) (\qii 0 - v_i)$,
  we conclude that
  \begin{equation}
    \label{eq:10}
    v_i - \qii \ell =  - \mu_i^{-\ell} (S_N^h)^\ell (I-P_h) (\qii 0 - v_i).
  \end{equation}

  {\em Step 3:} Since $S_N^h$ commutes with $P_h$,
  equation~\eqref{eq:10} implies
  \[
    \| v - \qii \ell \|_\cV 
    \le 
    \frac{1}{|\mu^h_i |^\ell}
    \left\| \big[S_N^h (I-P_h)\big]^\ell\right\|_\cV 
    \| v - \qii 0\|_\cV.      
  \]
  Let $\mu_*^h$ denote the supremum of $|\mu|$ over all $\mu$ in
  $\Sigma(S_N^h) \setminus \{\mu_1^h, \mu_2^h, \ldots, \mu_m^h\}$,
  i.e., $\mu_*^h$ is the spectral radius of $S_N^h(I - P_h)$, so 
  \[
    \mu_*^h = \lim_{\ell \to \infty} \| [ S_N^h (I - P_h) ]^\ell
    \|_\cV^{1/\ell}. 
  \]
  Hence, for any given $\veps>0$, there is an
  $\ell_0\ge 1$ such that
  $\| [ S_N^h (I - P_h) ]^\ell \|_\cV \le (\mu_h^* + \veps)^\ell$
  holds for all $\ell > \ell_0$ and consequently
  \begin{equation}
    \label{eq:7}
    \| v - \qii \ell \|_\cV 
    \le 
    \frac{ (\mu_*^h + \veps)^\ell }{|\mu^h_i|^\ell} \| v - \qii 0\|_\cV.
  \end{equation}

  {\em Step~4:} As already seen, a consequence of
  Assumptions~\ref{asm:rN} and~\ref{asm:Rlim}, is that by making $h$
  sufficiently small, we ensure that the eigenvalues
  $\mu_1^h, \mu_2^h,$ $\ldots,$ $\mu_m^h$ of $S^h_N$ are strictly
  separated in magnitude from the remaining eigenvalues --
  cf.~\eqref{eq:11}.  Hence we may choose an $\veps>0$ so small that
  $\delta_i = ( \mu_*^h + \veps)/ |\mu_i^h| <1.$ Then, with
  $\alpha_i = \| v_i - \qii 0\|_\cV,$ the estimate~\eqref{eq:7}
  implies
  \begin{equation}
    \label{eq:deltai}
    \| v_i - \qii \ell \|_\cV \le \alpha_i \delta_i^\ell, \qquad \ell
    > \ell_0.
  \end{equation}
  Note that $v_i,$ $i=1, \ldots, m$ form a basis for $E_h$.
  Hence, we may expand an arbitrary
  $v_h \in U^\cV_{E_h}$ in this basis and construct an approximation
  of $v_h$ using the same coefficients:
  \[ 
    v_h = \sum_{i=1}^m c_i v_i, \qquad q_\ell = \sum_{i=1}^m c_i
    q_i^{(\ell)}.
  \]
  Then, by~\eqref{eq:deltai}, 
  \begin{align}
    \label{eq:8}
    \dist {\cV} ( v_h, \Eh \ell) 
    \le \| v_h - q_\ell \|_{\cV} 
    \le \sum_{i=1}^m |c_i \alpha_i| \delta_i^\ell
    \le \alpha
    \left( \sum_{i=1}^m |c_i|^2\right)^{1/2}
    \left( \sum_{i=1}^m \delta_i^{2\ell}\right)^{1/2}.
  \end{align}
  where $\alpha = \max_i \alpha_i$.

  {\em Step~5:} Denote one of the two suprema in the definition of
  $\gap_\cV( E_h, \Eh \ell)$ by
  \[
  \delta_{h,\ell} = 
  \sup_{ v_h \in U^\cV_{E_h}} \dist {\cV} \left( v_h, \Eh \ell \right).
  \]
  Let $g$ denotes the minimal eigenvalue of the $m\times m$ Gram
  matrix of the $v_i$-basis (whose $(i,j)$th entry is
  $(v_i, v_j)_\cV$). Then
  $ g \sum_{i=1}^m |c_i|^2 \le \| v_h \|_\cV^2=1$. (Note that $g$ may
  depend on $h$, but is independent of $\ell$.)  Hence~\eqref{eq:8}
  implies
  $
  \delta_{h,\ell}^2
  \le (\alpha^2/g)  \sum_{i=1}^m \delta_i^{2\ell}
  $
  which converges to $0$ as $\ell \to \infty$ since $\delta_i <1$.

  In particular, for large enough $\ell$, we have
  $\delta_{h,\ell} <1$. Hence, by~\cite[Theorem~I.6.34]{Kato1995}
  there is a subspace $\Eht \ell \subseteq \Eh\ell$ such that
  $ \gap_\cV( E_h, \Eht\ell) = \delta_{h,\ell} <1.  $ Hence,
  $\dim(E_h) = \dim( \Eht \ell) =m$. But by~\eqref{eq:dimEhl}, the
  only subspace $\Eht\ell \subseteq \Eh\ell$ of dimension~$m$ is
  $\Eht\ell = \Eh\ell$. Thus, for sufficiently large $\ell$,
  \[
  \gap_\cV(E_h, \Eh\ell) = \delta_{h,\ell},
  \]
  and the proof is complete since $\delta_{h,\ell} \to 0$ as $\ell \to \infty$.
\end{proof}

To summarize this section, we have defined a space $E_h$
(in
Definition~\ref{def:Eh}) using $S_N^h$, but 
independently of the filtered subspace iteration~\eqref{eq:1},
and have shown (in
Theorem~\ref{thm:iteratesgap}) that under certain conditions the
iteration converges to it.  The convergence of FEAST iterations for
matrices (disregarding any discretization errors) were previously
studied in~\cite{GuttePolizTang15} when $\cH = \RRR^n$ and
$\| \cdot \|_{\cV}=\| \cdot\|_{\cH}$ using the theory of subspace
iterations~\cite{Saad16}.  In fact the identity obtained in Step~2 of
the above proof was motivated by a standard argument in the analysis
of subspace iterations~\cite{Saad16}. Our proof of
Theorem~\ref{thm:iteratesgap} gives a rigorous justification for the
intuition that if the discretization is good, then despite the errors
in the resolvent approximations, filtered subspace iteration should
converge for well-separated eigenvalue clusters.

\section{Discretization errors in eigenspace}
\label{sec:space}

In this section we study how the discrete eigenspace $E_h$ \revv{approaches}
the exact eigenspace $E$ as the discretization parameter $h$ goes to
$0$.

\begin{theorem}
  \label{thm:discrete}
  Suppose Assumptions~\ref{asm:rN}--\ref{asm:Rlim} hold.  Then there
  is a $C_N>0$ and an $h_0>0$ such that for all $h < h_0$,
  \begin{equation}
    \label{eq:gapEEh}
  \gap_\cV( E, E_h) \le C_N W \max_{k=1,\ldots, N} 
  \left\| \big[  R(z_k) - R_h(z_k) \big]
    \raisebox{-0.1em}{\ensuremath{\big|_E}} \right\|_\cV,       
  \end{equation}
  so, in particular, 
  \[
  \lim_{h\to 0} \gap_\cV (E, E_h) = 0.
  \]
\end{theorem}
\begin{proof}
  Consider one of the two suprema in the definition of
  $\gap_\cV(E_N,E_h)$, namely
  \begin{equation}
    \label{eq:deltaNh}
    \delta_h = \sup_{e \in U^\cV_{E_N}} \dist\cV ( e, E_h).
  \end{equation}
  Then,
  \begin{align}
    \label{eq:deltah}
    \delta_h
    & \le 
      \sup_{e \in U^\cV_{E_N}}  \| e - P_h e \|_\cV
      \le 
      \sup_{e \in U^\cV_{E_N}}  \| (P_N  - P_h) e \|_\cV.
  \end{align}
  Note that 
  \begin{align*}
    P_N - P_h 
    & = \frac{1}{2\pi \ii}    \oint_\Theta
    \left[ 
    ( z - S_N)^{-1} - (z- S^h_N)^{-1} \right] \, dz
    \\
    & 
    = \frac{1}{2\pi \ii}
    \oint_\Theta
    (z - S^h_N)^{-1} (S_N - S^h_N)  (z - S_N)^{-1} \, dz.
  \end{align*}
  Since $E_N$ is an invariant subspace of $ (z - S_N)^{-1}$, the above
  identity gives the estimate
  \begin{align*}
    \| P_N e- P_he \|_\cV
    & \le 
      \left[ 
      \frac{1}{2\pi }
      \oint_\Theta 
      \| (z - S_N)^{-1} \|_\cV \| (z - S^h_N)^{-1} \|_\cV \, dz
      \right]
      \| (S_N - S^h_N)|_{E_N} \|_{\cV} \| e\|_\cV.
  \end{align*}
  Returning to~\eqref{eq:deltah}, we conclude that
  $\delta_h \le C_N \| (S_N - S^h_N) |_{E_N} \|_{\cV}$, where $C_N$ is a
  bound for the quantity in square brackets above. Clearly, $C_N$ can
  be bounded independently of $h$, since
  $\| (z - S_N^h)^{-1} \|_\cV \to \| (z - S_N)^{-1} \|_\cV.$

  Thus, by virtue of~\eqref{eq:Shlim-V}, $\delta_h \to 0$ as $h \to 0.$
  In particular, for sufficiently small $h$, we have $\delta_h
  <1$. Then, by~\cite[Theorem~I.6.34]{Kato1995}, there is a closed
  subspace $\tilde E_h \subseteq E_h$ such that
  $\gap_\cV(E_N, \tilde E_h) = \delta_h <1$ and
  $\dim \tilde E_h = \dim E_N = m$. Because of~\eqref{eq:dimEh}, this
  implies that $\tilde E_h = E_h$.  Since $E_N = E$ by
  Lemma~\ref{lem:E=EN}, we finish the proof of~\eqref{eq:gapEEh} by
  noting that
  $ \gap_{\cV} ( E, E_h) = \gap_\cV( E_N, \tilde E_h) = \delta_h$.
\end{proof}

\begin{remark} \label{rem:gapH}
  If $\cV$ is replaced by $\cH$ in~\eqref{eq:gapV}, we obtain
  $\gap_{\cH}(M,L)$, so 
  \[
  \gap_\cH( E, E_h ) = \max \left[
    \sup_{e \in U_E^\cH} \dist{\cH} ( e, E_h),
    \sup_{m \in U_{E_h}^\cH} \dist{\cH} ( m, E)
  \right].
  \]
  Its natural to ask if $\gap_\cV( E, E_h) \to 0$ implies
  $\gap_\cH(E, E_h)\to 0$ as $h \to 0$.  Let $\delta^\cH_h $ denote
  the first of the two suprema above.
  Since $E$ is
  finite-dimensional, there is a $C_m>0$ such that
  $ \| e \|_{\cH} \ge C_m\| e \|_{\cV}$ {for all $e$ in $E$.}
  Using also  Assumption~\ref{asm:Vshort},
  \[
  \delta^\cH_h = \sup_{0 \ne e \in E} \frac{\dist\cH( e, E_h) }{ \| e \|_{\cH}}
  \le 
  \,\frac{C_\cV}{C_m}\, 
  {\sup_{0 \ne e \in E}}\frac{ \dist\cV( e, E_h) }{  \| e \|_{\cV}}
  \le 
  \,\frac{C_\cV}{C_m}\, \gap_\cV(E,E_h).
  \]
  Thus, if $\gap_\cV( E, E_h) \to 0$, taking $h$ sufficiently small,
  $\dim(E_h) = \dim(E) = m$ and $\delta^\cH_h <1$, so
  using~\cite[Theorem~I.6.34]{Kato1995} as in the previous proof, we
  may conclude that $\gap_\cH (E, E_h) = \delta_h^\cH$.  This implies
  that, under the same assumptions as in Theorem~\ref{thm:discrete},
  there is an $h_1>0$ such that
  \begin{equation}
    \label{eq:gapHgapV}
    \gap_\cH(E, E_h) \le 
    \,\frac{C_\cV}{C_m}\, \gap_{\cV}(E, E_h)
  \end{equation}
  for all $h < h_1$.
  {Note $C_m$ depends only on $E$ and  is independent of $h$.}
\end{remark}

\section{Discretization errors in eigenvalues}
\label{sec:ew}

In this section, we analyze the eigenvalue approximations that are
generated as Ritz values (defined below) of eigenspace approximations
obtained from the filtered subspace iteration. To define the Ritz values
maintaining the same level of generality as we have so far, we need to
consider the (possibly unbounded) sesquilinear form generated by~$A$.

Recall that any selfadjoint operator $A$ admits the polar
decomposition $A = U_A |A| = |A| U_A$ (see~\cite[p.~335]{Kato1995}),
where $U_A$ is selfadjoint and partially isometric, and $|A|$ is
selfadjoint and positive semidefinite.  As described
in~\cite[\S10.2]{Schmu12}, the polar decomposition can be used to
define the following symmetric sesquilinear {\em
  form associated to the operator} $A$:
\begin{equation}
  \label{eq:6}
  a(x,y) = (U_A|A|^{1/2}x, |A|^{1/2}y)_\cH  
\end{equation}
for any $x,y$ in $\dom(a) = \dom(|A|^{1/2})$. 
Let
$| u |_a= |a(u,u)|^{1/2}$ for any $u \in \dom(a)$.
By the properties of $U_A$, 
\begin{equation}\label{ineq||}
| u |_a \leq (|A|^{1/2}u,|A|^{1/2}u)^{1/2}=\|
|A|^{1/2}u\|_{\mathcal{H}},\qquad u\in \dom(a).
\end{equation}

Let $F \subset \dom(a)$ be a closed finite-dimensional subspace of
$\cH$. We define $A_F:F\to F$ by the relation 
\begin{equation}
  \label{eq:12}
  (A_F x,y)_\cH =a(x,y) \qquad \text{ for all } x, y \in F.
\end{equation}
The spectrum of the linear operator $A_F$ on $F$, namely
$\Sigma(A_F)$, is called the set of {\em Ritz values of $A$ on $F$}.
The operator $A_E$ is defined by~\eqref{eq:12} with $E$ in place of
$F$.  Note that the exact eigenspace we wish to approximate, namely
$E,$ is contained in $ \dom(A) \subset \dom(a)$, and the exact
eigenvalue cluster $\Lambda$ that we wish to approximate is the set of
Ritz values of $A$ on~$E$.

Central to the discussion of this section is how the Ritz values
change when $F$ is a perturbation of $E$. To formulate a result on
sensitivity of Ritz values, we need more notation.  Recall that $S$ is
the $\cH$-orthogonal projection onto~$E$.  Let $Q$ denote the
$\cH$-orthogonal projection onto $F$.  Using $S$, we may express $A_E$
as
\[
A_E=S|A|^{1/2}U_A|A|^{1/2}S\big|_E
\]
and $A_F$ may be similarly expressed using $Q$. Let 
\begin{equation}
  \label{eq:14-F}
| (S -I)Q |_{a, F} = \sup_{0 \ne v \in F} 
\frac{ |(S - I)Q v |_a}{ \| v \|_\cH}
=
\sup_{0 \ne v \in F } 
\frac{ |(S - I) v |_a}{ \| v \|_\cH}.  
\end{equation}
Note that there is a finite positive constant
$\|A_E\|$ such that $\| A_E e \|_\cH \le \| A_E \| \| e\|_\cH$ for all
$e \in E$ (since $A_E: E \to E$ and $E$ is finite-dimensional).
Define the Hausdorff distance between two subsets $\Upsilon_1, \Upsilon_2
\subset \RRR$ by
\[
\hdist( \Upsilon_1, \Upsilon_2) = 
\max 
\left[
\sup_{\mu_1 \in \Upsilon_1} \hdist( \mu_1, \Upsilon_2), 
\sup_{\mu_2 \in \Upsilon_2} \hdist( \mu_2, \Upsilon_1)
\right]
\]
where $\hdist(\mu, \Upsilon) = \inf_{\nu \in \Upsilon} | \mu - \nu|$ for any 
$\Upsilon \subset \R.$  \revv{The following lemma is a perturbation
  result that can be understood  independently of the filtered
  subspace iteration. In particular, we have no need for 
  Assumptions~\ref{asm:rN}--\ref{asm:Rlim} in the lemma.}

\begin{lemma}
  \label{lem:Ritz}
  Suppose $\gap_\cH(E, F) < 1$. Then there is a $C_0 >0$ such that 
  \[
  \hdist( \Sigma(A_E), \Sigma(A_F)) 
  \le 
  | (S-I) Q |_{a, F}^2 + 
  C_0 \| A_E \| \, \gap_\cH(E, F)^2.
  \]
\end{lemma}
\begin{proof}
  {\em Step 1:} 
  Let  $R = (S -Q)^2$ and let 
   $\delta <1 $ be any number satisfying 
  $\gap_\cH(E, F) \le \delta <1.$ Since
  $\| R \|_\cH \le \gap_\cH( E, F)^2 \le \delta^2 < 1$, the binomial
  series  $ \sum_{n=0}^\infty (\begin{smallmatrix} -1/2 \\ n
  \end{smallmatrix}) (-R)^n $ converges and defines $
  (I-R)^{-1/2}$. Subtracting the first term from this series, define 
  $T = (I-R)^{-1/2} - I$. Since
  $(1-x)^{-1/2} -1 = x [\sqrt{1-x} + (1-x)]^{-1}$, we obtain that
\begin{align*}
\|T\|_\cH\leq \sum_{n=1}^\infty \binom{-1/2}{n}\|R\|_{\cH}^n =(1-\|R\|_{\cH})^{-1/2} -1 = \|R\|_{\cH} [\sqrt{1-\|R\|_{\cH}} + (1-\|R\|_{\cH})]^{-1}~,
\end{align*}
 which implies
  \begin{equation}
    \label{eq:Tbd}
    \| T\|_\cH \le \| R \|_\cH \left[ \sqrt{1- \delta^2} + 
      (1 - \delta^2) \right]^{-1}.
  \end{equation}
  We use $R$ to define an isometry
  $J = (I-R)^{-1/2} [ QS + (I-Q)(I-S)] $ on $\cH$
  (cf.~\cite[p.~33]{Kato1995}) which maps $E$ one-to-one onto $F$,
  and whose inverse is
  \begin{equation}
    \label{eq:Uinv}
    J^{-1} = J^* = \left[ SQ + (I-S)(I-Q) \right] (I-R)^{-1/2}.
  \end{equation}
  Note that the spectra of $A_E$ and the unitarily equivalent
  $ JA_E J^*|_{{F}}$ are identical.

  {\em Step 2:} Let $D = JA_E J^*|_{F} - A_{F},$ a selfadjoint operator
  on $F$.  By~\cite[Theorem~V.4.10]{Kato1995},
  \begin{equation}
    \label{eq:hdorf}
    \hdist( \Lambda, \Lambda_h) \le \| D|_{F} \|_\cH = 
    \sup_{0 \ne f \in F} 
    \frac{ | (Df,f)_{\cH} |}{(f,f)_\cH}.
  \end{equation}
  For $f\in F$, we have
\begin{align*}
(Df,f)_\cH &= a( J^* f, J^* f) - a(f,f) =
             a( S J^* f, SJ^* f) - a(Qf,Qf) 
  \\
           & = 
             \Re a (SJ^* f + Q f,  SJ^* f - Q f).
\end{align*}
 Observing that~\eqref{eq:Uinv} implies 
  $S J^* = SQ (I-R)^{-1/2}$,  we split
  \begin{align*}
    (Df,f)_\cH 
    & = 
      \Re a ( (SJ^*  + Q) f, SQ \left[ (I-R)^{-1/2} - I\right]f)
    \\
    & +
      \Re a ( (SJ^*  + Q) f, (S-I)Q    f).
  \end{align*}
  Labelling the two terms on the right as $t_1$ and $t_2,$ we proceed
  to estimate them.

  {\em Step 3:} The first term
  $t_1 = \Re a ( (SJ^* + Q) f, SQ T f) = \Re a ( (SJ^* + S Q) f,
  SQ T f)$.
  Here we have used $S^2 = S$ and $a(S x, y) = a(x, Sy)$ for all
  $x, y \in \dom(a)$. The latter follows from~\eqref{eq:6} because $S$
  commutes with $A$, so it commutes with $|A|$ and $U$
  \cite[p.335ff]{Kato1995}, and moreover, it commutes with $|A|^{1/2}$
  (see e.g.~\cite[Theorem~V.3.35]{Kato1995}). Continuing, 
  \begin{align*}
    |t_1|
    & =  |\Re a ( (SJ^*  + Q) f, SQ T f)| 
      = 
      |\Re (A_E S (J^* + Q) f, SQTf )_\cH|
      \\
    & \le
      \| A_E\| \| S (J^* + Q) f \|_\cH \| SQTf \|_\cH
     \le 
      \frac{ 2 \| A_E\| \|R\|_\cH }{ \sqrt{1- \delta^2} + 
      (1 - \delta^2) } \| f \|_\cH^2
  \end{align*}
  where we have used the fact that orthogonal projectors have unit
  norm as well as the isometry of $J^*$.  Thus
  \begin{align*}
    | t_1| 
    & \le  
      \frac{2 \|A_E\|}{ \sqrt{1- \delta^2} + 
      (1 - \delta^2) } \gap_\cH(E,F)^2 \| f \|_\cH^2.
  \end{align*}

  {\em Step 4:} Next, we estimate $t_2$. Since
  $ a((S-I)x, y) = a(x, (S-I)y)$ for all $x, y \in \dom(a)$,
  \begin{align*}
    |t_2| 
    & = | \Re a ( (S-I) (SJ^*  + Q) f, (S-I)Q    f)| 
    \\
    & = |\Re a ( (S-I) Q f, (S-I)Q    f)| 
    \\
    & \le | (S-I)Q |_{a, h}^2 \| f \|_\cH^2.
  \end{align*}
  Adding the estimates for $|t_1|$ and $|t_2|$ and using it
  in~\eqref{eq:hdorf}, the proof is finished.
\end{proof}

Before applying this lemma to filtered subspace iteration, a few
remarks are in order. 
{\em (i)}~Its clear from the proof that the result of the lemma holds
even when dimension of $E$ (and $F$) is infinite, as long as
$\| A_E\| < \infty.$ Its also clear from the proof that the constant
$C_0=2 /({\sqrt{1- \delta^2} + 1 - \delta^2})$ is independent of the
location of eigenvalue cluster $\Lambda$.
{\em (ii)} The quantity $| (S-I) Q|_{a, F}^2$ is related to the square of
  the gap (like the other term in the bound of Lemma~\ref{lem:Ritz}).
  Indeed, if $\caF$ is any constant that satisfies
  $ |v|_a^2 \le \caF \| v \|_\cH^2$ for all $v \in E + F,$ then
  \begin{equation}\label{H:estimate} 
    | (S-I) Q|_{a, F}^2 \le \caF \|
    (S-I)Q\|_\cH^2 \le \caF \gap_\cH(E, F)^2.
  \end{equation} 
  However, in applications, we usually need to make the dependence of
  $\caF$ more explicit (say, on discretization parameters).  One
  technique for this is developed in the proof of
  Corollary~\ref{EigenvalueError_Gaps|A|} below.
{\em (iii)}~We highlight that Lemma~\ref{lem:Ritz} applies to {\em general
  unbounded} selfadjoint operators, even those whose spectra 
  extends throughout the real line.
  {\em (iv)}~Bounds for the Hausdorff distance between Ritz values
  under space perturbations have been previously studied for {\em
    bounded} operators~\cite[Theorem~5.3]{Knyazev2} and a part of the
  above proof above is inspired by their arguments. However we are not
  able to use their result directly because it holds only for Ritz
  values located at the extremes (top or bottom) of the spectrum of
  the bounded operator. Nonetheless, an approach to bring
  \cite[Theorem~5.3]{Knyazev2} to bear on unbounded operators is to
  apply it to $R(\mu) = (\mu - A)^{-1}$, which is bounded (even if $A$
  is unbounded) provided $\mu$ is in the resolvent set of $A$.  To
  quickly sketch this approach, one choses a $\mu$ such that $E$
  becomes the eigenspace of $R(\mu)$ corresponding to the top of its
  spectrum, then apply \cite[Theorem~5.3]{Knyazev2} to obtain an
  estimate that bounds the distance between Ritz values of $R(\mu)$,
  from which one then concludes estimates on  the distance between
  Ritz values of $A$ on $E$ and on $F$. This technique would yield bounds
  involving $\gap_\cH(E, F)^2$ like that of Lemma~\ref{lem:Ritz} but
  with other $\mu$-dependent constants.
  {\em (v)} For finite-dimensional $E, F,$ perturbations in
  eigenvalues of bounded operators corresponding to the top or bottom
  of the spectrum have also been studied in
  \cite[Theorem~2.7]{KnyazevFEM} using majorization techniques. Their
  estimates can also be used to study spectral perturbations of
  unbounded operators by the technique mentioned in item {\em (iv)}
  above. In cases where one can bound specific angles between $E$ and
  $F$, then \cite[Theorem~2.7]{KnyazevFEM} may provide bounds for
  individual eigenvalue errors that are sharper than what can be
  concluded from bounds of the Hausdorff distance.

We now turn to the issue of approximating the eigenvalue cluster
$\Lambda$ 
using the subspaces of $\cV_h$ generated by the filtered subspace iteration
using $S_N^h.$ Our analysis of this approximation is under the next
assumption.  Example~\ref{eg:Neumann} below illustrates the reason to
consider forms and place this assumption.

\begin{assumption}
  \label{asm:sqrootmodA}
  Assume that $\cV_h$ is contained in $\dom(a)$.
\end{assumption}

\begin{example}[Positive operators]
  \label{eg:positiveform}
  Consider the operator $A$ and the form $a$ in
  Example~\ref{eg:case:B}. Here, since $A$ is positive, the factors of
  the polar
  decomposition of $A$ are $U_A = I$ and $|A|=A$. Thus
  $\dom(a) = \dom( |A|^{1/2} ) = \dom( A^{1/2})$. 
  Moreover, $\cV = \dom(A^{1/2})$ in Example~\ref{eg:case:B}.  Since
  $\cV_h \subset \cV$ by definition, we conclude that
  Assumption~\ref{asm:sqrootmodA} holds.~\hfill$\Box$
\end{example}

\begin{example}[A differential operator]
  \label{eg:Neumann}
  To give an example of a partial differential operator fitting the
  scenario of Example~\ref{eg:positiveform}, suppose $\om$ is an open
  subset of $\RRR^d,$ $\beta:\om \to \RRR$ is a bounded positive
  function, and $\alpha: \om \to \CCC^{d \times d}$ is a bounded
  Hermitian positive definite matrix function. Suppose the smallest
  eigenvalue of $\alpha(x)$ and $\beta(x)$ are greater than some
  $\delta>0$ for a.e.\ $x\in \om$. Put $\cH = L^2(\om)$ and set $a$ by
  \begin{equation}
    \label{eq:4}
  a(u,v) = \int_\om \alpha \grad u \cdot \grad \overline{v}  \; dx
  + \int_\om  \beta u \overline{v}  \; dx    
  \end{equation}
  for all $u,v$ in $\dom(a) = H^1(\om)$. This is a densely defined
  closed form. Set $A$ to be the closed selfadjoint {\em operator
    associated to the form} $a$, obtained by a representation
  theorem~\cite[Theorem~10.7]{Schmu12}. 

  When $\alpha$ and $\beta$ equal the identity and $\om$ has Lipschitz
  boundary, the operator $A$ is a Neumann operator whose domain
  satisfies $ \dom(A) \subseteq H^{3/2}(\om) $ by a result
  of~\cite{JerisKenig81}. Thus $\dom(A)$ is strictly smaller than
  $ \dom(a) = \dom(A^{1/2}) = H^1(\om)$ in this case.  Therefore, if
  $\cV_h$ is set to the Lagrange finite element subspace of
  $H^1(\om)$, then Assumption~\ref{asm:sqrootmodA} holds.  Note that
  it is easier to build finite element subspaces of $H^1(\om)$ than
  $H^{3/2}(\om)$, which is why we did {\em not} require $\cV_h$ to be
  contained in $\dom(A)$ in
  Assumption~\ref{asm:sqrootmodA}.~\hfill$\Box$
\end{example}

\begin{example}[Semibounded operators]
  \label{eg:semibounded}
  Suppose $A$ is {\em lower semibounded}, i.e., there is a
  $\mu \in \RRR$ such that $(A x,x)_\cH \ge \mu\,(x,x)_\cH$ for all
  $x$ in $\dom(A)$. Then, by \cite[Proposition~10.5]{Schmu12}, 
  \begin{equation}
    \label{eq:5}
    \dom(|A|^{1/2}) = \dom( (A - \mu)^{1/2} ).
  \end{equation}
  
  An example of such an operator is the operator associated to the
  form $a$ in~\eqref{eq:4} when $\beta$ no longer satisfies
  $\beta >0$, but instead changes sign while remaining bounded on
  $\om$. Then fixing some $\mu < -\|\beta\|_{L^\infty(\om)},$ we note
  that the operator $A - \mu$ is positive and is the operator
  associated with the positive form
  $a_\mu(u,v) = a(u,v) - \mu(u,v)_\cH$. Thus, by
  Example~\ref{eg:positiveform},
  $\dom(a_\mu) = \dom( (A-\mu)^{1/2}) = H^1(\om)$. Hence
  by~\eqref{eq:5} we conclude that $\dom(a) = H^1(\om)$.~\hfill$\Box$
\end{example}

\begin{remark}
  Above we have encountered two related, but distinct concepts, of the
  {\em form associated to an unbounded operator} (via the polar
  decomposition as in~\eqref{eq:6}) and the {\em operator associated
    to a unbounded form} (by the first representation theorem
  \cite[TheoremVI.2.1]{Kato1995}). If one begins with a form $a$ and
  then considers the operator $A$ associated to it, we can define
  another form $\tilde{a}$ that is the form associated to $A$. The
  form $\tilde{a}$ need not equal $a$ for a general selfadjoint
  operator as shown
  in~\cite[Example~2.11]{GrubiKostrMakar13}). However, $a$ and
  $\tilde a$ are equal if $a$ is a densely defined lower semibounded
  closed form by \cite[Theorem~10.7]{Schmu12}.
\end{remark}

With the above background in mind, we now return to the analysis of
eigenvalue approximations.  Recall
$E_h\subset \cV_h \subset \dom(|A|^{1/2})$, the space we studied in
Section~\ref{sec:limit}).  Using $E_h$, we compute the spectrum of the
finite-dimensional operator $A_{E_h}$,
\[
\Lambda_h = \Sigma(A_{E_h}).
\]
This set forms our approximation to $\Lambda$. 
In practice, the elements of $\Lambda_h$ 
are computed by solving a
small dense generalized eigenproblem arising from an equivalent equation
of forms: find $\lambda_h\in\RR$ and $0\neq u_h\in E_h$ satisfying
\[
a(u_h, v_h) = \lambda_h(u_h, v_h)_{\cH}
\]
for all $v_h \in E_h$. The collection of all such $\lambda_h$ forms
$\Lambda_h$.  In the next theorem, we use $Q_h$ to denote the
$\cH$-orthogonal projection onto $E_h$.

\begin{theorem}
  \label{thm:HausDistEWs}
  Suppose Assumptions~\ref{asm:rN}--\ref{asm:sqrootmodA} hold.  Then
  there are positive constants $C_0$ and $h_0$ such that for all
  $h < h_0$,
  \[
    \hdist( \Lambda, \Lambda_h) \le 
    | (S-I) Q_h |_{a, E_h}^2 + 
    C_0 \| A_E \| \, \gap_\cH(E, E_h)^2.
  \]
\end{theorem}
\begin{proof}
  By Theorem~\ref{thm:discrete} and~\eqref{eq:gapHgapV} we may choose
  $h$ so small that $\gap_\cH(E, E_h) \le \delta <1.$ Hence, applying
  Lemma~\ref{lem:Ritz}, the result follows.
\end{proof}

\begin{corollary}\label{EigenvalueError_Gaps|A|}
  In addition to Assumptions~\ref{asm:rN}--\ref{asm:sqrootmodA},
  suppose $\|u\|_{\cV}=\||A|^{1/2} u\|_{\cH}$. Then 
  there are positive constants $C_0$ and $h_0$ such that for all
  $h < h_0$, \revv{we have $\gap_\cV(E, E_h) <1$ and}
  \[
  \hdist( \Lambda, \Lambda_h) \le (\lmax)^2 \gap_\cV(E, E_h)^2
  + 
  C_0 \| A_E \| \, \gap_\cH(E, E_h)^2\revv{,}
  \]
  where $\lmax = \sup_{e_h \in E_h} \| |A|^{1/2} e_h\|_{\cH} / \| e_h
  \|_\cH$ \revv{satisfies 
  \[
  (\lmax)^2\le \left(\frac{1}{1-\gap_\cV(E, E_h)}\right)^2  \| A_E\|.
  \]}
\end{corollary}
\begin{proof}
 The first inequality 
 will follow from Theorem~\ref{thm:HausDistEWs} by establishing that
 \[ 
 |(S - I) Q |_{a, E_h} \le \lmax \gap_\cV(E, E_h).
 \]
 Since $S$ is a $\cH$-orthogonal projection, it is selfadjoint in $\cH$-inner
 product. Moreover, since $S$ commutes with $A$, it commutes with
 $|A|$ and hence with $|A|^{1/2}$. Therefore,
 \begin{align*}
   (S u, v)_\cV
   & = (|A|^{1/2} S u, |A|^{1/2} v)_\cH 
     = (S |A|^{1/2}  u, |A|^{1/2} v)_\cH = 
     (|A|^{1/2}  u, S |A|^{1/2} v)_\cH = (u, S v)_\cV,
 \end{align*}
 \revv{for all $u, v \in \cV$,}
 i.e., $S$ is selfadjoint in the $\cV$-inner product too. This implies
 that $S$ is also the $\cV$-orthogonal projector onto $E$. Hence,
 using~\eqref{ineq||},
 \begin{equation}
   \label{eq:13b}
   |S e_h - e_h |_a\leq \|S e_h - e_h\|_{\cV} = \dist{\cV}( e_h, E)
 \end{equation}
 for any $e_h \in E_h$.
 Combining~\eqref{eq:14-F} with~\eqref{eq:13b}, we  conclude that 
 \begin{align*}
   | (S -I)Q |_{a, E_h} 
   & \le
     \left( 
     \sup_{0 \ne e_h \in E_h} \frac{ \|e_h\|_{\cV}} {\| e_h\|_{\cH}}
     \right) 
     \left( \sup_{0 \ne e_h \in E_h } 
     \frac{ |(S - I) e_h |_a}{ \| e_h \|_\cV}
     \right)
   \\
   & \le  \lmax
     \left( \sup_{0 \ne e_h \in E_h } 
     \frac{ \dist{\cV} (e_h, E)}{ \| e_h \|_\cV}
     \right)
 \end{align*}
 \revv{The first inequality of the corollary now follows from~\eqref{eq:gapV}.}

 \revv{
 Let $g = \gap_\cV(E, E_h)$ and $e_h \in E_h$. Then~\eqref{eq:13b} implies
 \begin{align}
   \label{eq:13}
 \|S e_h \|_\cV 
   &  \ge 
     \| e_h \|_\cV - \| e_h - Se_h \|_\cV 
     = \| e_h \|_\cV 
     \left( 1 - \frac{ \dist{\cV} (e_h, E)}{\| e_h\|_{\cV} }\right)
     \ge \| e_h \|_\cV 
     \left( 1 - g\right).
 \end{align}
 Therefore, 
 \begin{align*}
   \frac{\| e_h\|_\cV}{ \| e_h\|_\cH} 
   & 
     = \frac{\| e_h\|_\cV}{ \|S e_h\|_\cV} 
     \;\frac{\| Se_h\|_\cV}{ \| e_h\|_\cH} 
     \le 
      \frac{1}{1-g}     
     \;\frac{\| Se_h\|_\cV}{ \| e_h\|_\cH}.
 \end{align*}
 Since
 $\| S e_h \|_\cV^2 = |(|A| Se_h, Se_h)_{\cH}| = |(U_A A Se_h,
 Se_h)_{\cH} | \le \| A_E\| \| Se_h \|_\cH^2 \le \| A_E\|\| e_h
 \|_\cH^2$,
 \begin{align*}
   \lmax 
   &= \sup_{0 \ne e_h \in E_h} \frac{\| e_h\|_\cV}{ \| e_h\|_\cH} 
     \le 
      \frac{1}{1-g}     \;
     \sup_{0 \ne e_h \in E_h}\frac{\| Se_h\|_\cV}{ \| e_h\|_\cH}
      \le 
     \frac{1}{1-g}  \| A_E\|^{1/2}.
 \end{align*}
 This completes the proof.}
\end{proof}

\revv{Note that the second inequality of
  Corollary~\ref{EigenvalueError_Gaps|A|} allows one to bound $\lmax$
  independently of $h$ when $\gap_\cV(E, E_h) \to 0$.} 
A class of examples where Corollary~\ref{EigenvalueError_Gaps|A|}
immediately applies is given by the positive forms of
Example~\ref{eg:positiveform}. For such operators, we have
$|A|^{1/2} = A^{1/2}$, so
$(u, v)_\cV = a(u,v) = (A^{1/2} u, A^{1/2} v)$ holds for all
$u, v \in \dom(a)$ and Corollary~\ref{EigenvalueError_Gaps|A|}
applies.  As a final remark, we also note that in the case in which
$\cV$ is normed with a norm equivalent to $\||A|^{1/2}\cdot\|_{\cH}$,
the above argument provides the estimate
\begin{equation}
\hdist( \Lambda, \Lambda_h) \le (C_1\,\Lambda_h^{\max})^2 \gap_\cV(E, E_h)^2
+ 
C_0 \| A_E \| \, \gap_\cH(E, E_h)^2.    
\end{equation}
where $C_1$ depends on the equivalence constants for norms $\|\cdot\|_\cV$ and $\||A|^{1/2}\cdot\|_{\cH}$.

\section{Application to the discretization of a model operator}
\label{sec:fem}

The purpose of this section is to provide an example for application
and illustration of the theoretical framework developed in the
previous sections.  A simple model problem is obtained by setting
\[
{  \cH = L^2(\om), \quad
  \cA = -\Delta, \quad
  \dom(\cA) = \{ \psi\in H_0^1(\om):~\Delta\psi\in L^2(\om)\} ,\quad
  \cV = H_0^1(\om),}
\]
{where $\om$} $\subset \RR^d$ ($d=2,3$) is a bounded polyhedral
domain with Lipschitz boundary.  Note that here $\| u \|_\cV$ is set
to $\| A^{1/2} u \|_\cH = \| \grad u \|_{L^2(\om)} = |u|_{H^1(\om)}$,
which is equivalent to $H^1(\om)$-norm due to the boundary condition.
Throughout, we use standard notations for norms ($ \| \cdot \|_X$) and
seminorms ($|\cdot|_X$) on Sobolev spaces~($X$).  The above set
operator $A$ is the operator associated to the form
\[
a(u,v) = \int_\om \grad u \cdot \grad \overline{v} \; dx, \quad 
u, v \in \dom(a) = \cV = H_0^1(\om).
\]
Hence this model problem fits into Example~\ref{eg:case:B} and
Assumption~\ref{asm:Vshort} holds.

To calculate the application of the resolvent $u = R(z) v$, we need to
solve the operator equation $(z - \cA ) u = v$ for any 
$z$  in the resolvent set of $A$. Using the form 
 $
b_z(u,v) = z(u,v)_\cH - a(u,v),
$
this
equation may be stated in weak form as the problem of finding
$R(z) v \in H_0^1(\om) \subset \cH$ satisfying
\begin{equation}
  \label{eq:Rzv-weak}
	b_z(R(z)v,w) = (v,w)_\cH
\qquad \text{ for all } w\in H_0^1(\om).
\end{equation}
\revv{
We provide a bound on the biliear form $b_z$ that will be used in
subsequent analysis.
\begin{lemma}\label{bzBound}
It holds that
\begin{align*}
|b_z(u,v)|\leq \alpha(z)|u|_{H_0^1(\om)}|v|_{H_0^1(\om)}\qquad\mbox{for all
  }u,v\in H_0^1(\om)~,
\end{align*}
where $\alpha(z)=\|zA^{-1}-I\|_\cH=\sup\{|\lambda-z|/|\lambda|:\,\lambda\in\Sigma(A)\}$.
\end{lemma}
\begin{proof}
Recognizing that $b_z(u,v)=((zA^{-1}-I)A^{1/2}u,A^{1/2}v)_\cH$, we see
that
\begin{align*}
|b_z(u,v)|\leq \|zA^{-1}-I\|_\cH \|A^{1/2}u\|_\cH \|A^{1/2}v\|_\cH=\alpha(z) |u|_{H_0^1(\om)}|v|_{H_0^1(\om)}~.
\end{align*}
Since $zA^{-1}-I$ is normal, its $\cH$-norm, $\alpha(z)$,  is equal to
its spectral radius.  
Therefore,
$\alpha(z)=\sup\{|\lambda-z|/|\lambda|:\,\lambda\in\Sigma(A)\}$, as claimed.
\end{proof}
The quantity
\begin{equation}
  \label{eq:beta}
  \beta(z)=\|AR(z)\|_\cH=\sup\{|\lambda|/|\lambda-z|:\,\lambda\in\Sigma(A)\}  
\end{equation}
also figures in our
analysis.  The last equality of the $\cH$-norm and the spectral radius
again follows from the normality of $AR(z)$. 
It can also be found  in~\cite[p. 273, Equation (3.17)]{Kato1995}.
Since $\Sigma(A)$ is
real, we see that $\alpha(\bar{z})=\alpha(z)$ and $\beta(\bar{z})=\beta(z)$.
}

Next, suppose $\om$ is partitioned by a conforming simplicial finite
element mesh $\oh$, where $h$ equals the maximum of the diameters of
all mesh elements. We shall assume that $\oh$ is shape regular and
quasiuniform (see e.g.,~\cite{Ern2004} for standard finite element
definitions). Set $\cV_h$ equal to the Lagrange finite element
subspace of $H_0^1(\om)$ consisting of continuous functions, which
when restricted to any mesh element $K$ in $\oh$, is in $P_{p}(K)$ for
some $p\ge 1$.  Here and throughout $P_\ell(K)$ denote the set of
polynomials of total degree at most~$\ell$ restricted to $K$.  It is
well known \cite{Ern2004} that there is a $C_{\text{app}}>0$
independent of $h$ such that
  \begin{equation}
    \label{eq:approxFEM}
    \inf_{\nu_h \in \cV_h} \revv{| \nu -   \nu_h |_{H^1(\om)}}
    \le 
    C_{\text{app}}
    h^r |\nu|_{H^{1+r}(\om)}
  \end{equation}
  for any $0\le r \le p$ and any $\nu \in H^{1+r}(\om)$.

Consider any $v \in \cH = L^2(\om)$.  The approximation of the
resolvent by the finite element method, namely $R_h(z) v$, is a function  in 
$\cV_h$ satisfying
\begin{equation}
  \label{eq:Rzv-weak-h}
    b_z(R_h(z) v, w) = (v, w) 
    \qquad \text{ for all } w\in \cV_h.
\end{equation}
It will follow from the ensuing analysis that~\eqref{eq:Rzv-weak-h}
uniquely defines $R_h(z) v $ in $\cV_h$ provided $h$ is sufficiently
small. The analysis is under the following regularity assumption.
\begin{assumption}
  \label{asm:reg}
  Suppose there are positive constants $\Creg$ and $s$ such that the solution
  $u^f \in \cV$ of the Dirichlet problem $-\Delta u^f = f$ admits the
  regularity estimate
\begin{equation}
	\label{eq:reg}
	\| u^f \|_{H^{1+s} (\om)} \le \Creg \| f \|_\cH
        \quad\text{ for any } f \in \cV.
\end{equation}
Suppose also that there is a number $s_E \ge s$ such that 
\begin{equation}
	\label{eq:reg-eig}
	\| u^f \|_{H^{1+{s_E}} (\om)} \le \Creg \| f \|_\cH
        \quad\text{ for any } f \in E.
\end{equation}
\end{assumption}
Standard regularity results for elliptic operators (see,
e.g.~\cite{Grisvard1992,Grisvard2011}) yield that
$\dom(A)\supset H^{1+s}(\om)$ for some $s>0$ depending on the geometry
of $\om$.  For example, if $\om$ is convex, we may take $s=1$
in~\eqref{eq:reg}; and if $\om\subset\RR^2$ is non-convex, with its
largest interior angle at a corner being $\pi/\alpha$ for some
$1/2<\alpha<1$, we may take any positive $s<\alpha$.
One can often show higher regularity when $f$ is restricted to the
eigenspace~$E$, which is why we additionally
assume~\eqref{eq:reg-eig}.  For example, if $\om=(0,1)\times(0,1)$,
all eigenfunctions are analytic, having the form
$\sin(m\pi x)\sin(n\pi y)$, for any positive integers $m,n$.  These
expressions, when viewed as functions on the L-shaped hexagon
$\om_{\mathrm{L}}=(0,2)\times(0,2)\setminus [1,2]\times[1,2]$, also
yield smooth eigenfunctions of $\om_{\mathrm{L}}$. But not all
eigenfunctions of $\om_{\mathrm{L}}$ are so regular.

\revv{
The proof of the next result is modeled after a  
classical argument of~\cite{Schat74}, and employs the quantities
$\alpha(z)$ and $\beta(z)$ introduced above.
}

\begin{lemma} \label{lem:R-Rh:Dir}
  Suppose Assumption~\ref{asm:reg} holds. Let $z$ be in the resolvent
  set of $A$.  Then there are positive constants $C$ and $h_0$
  (depending on $z$) such that for all $h < h_0$
  \begin{align}
    \label{eq:R-R_h-V}
    \| R(z) - R_h(z)  \|_\cV 
    & \le C h^r, 
    &
    \left\| \big[  R(z) - R_h(z) \big]
    \raisebox{-0.1em}{\ensuremath{\big|_E}} \right\|_\cV
    & \le C h^{r_E}, 
    \\
    \label{eq:R-R_h-H}
    \| R(z) - R_h(z)  \|_\cH
    & \le C h^{2r},
    &
    \left\| \big[  R(z) - R_h(z) \big]
    \raisebox{-0.1em}{\ensuremath{\big|_E}} \right\|_\cH
    & \le C h^{r_E + r},
  \end{align}
  where $r = \min(s, p)$ and $r_E = \min(s_E, p)$. \revv{We may choose
  $h_0 = [2\alpha(z)\beta(z) |z| \Capp \Creg]^{-1/r}$ and
  $C = 2\alpha(z)\beta(z) \Capp \Creg.$}
\end{lemma}
\begin{proof}
  Let $f \in \cH$, $e = R(z) f- R_h(z)f$, and $w = R(\overline{z}) e$.
\revv{
  Then $-\Delta w = A w = A R(\overline{z}) e$. Hence 
it follows by Assumption~\ref{asm:reg} and~\eqref{eq:beta} that
  \begin{align}
    \label{eq:w_Hs_bound}
    \| w \|_{H^{1+r}(\om)} 
    & \le \Creg \| AR(\overline{z}) e\|_\cH\leq
      \Creg\beta(\overline{z})\|e\|_\cH=
      \Creg\beta(z)\|e\|_\cH.
  \end{align}
}%
  Subtracting~\eqref{eq:Rzv-weak-h} from~\eqref{eq:Rzv-weak}, we have
  $b_z( e, w_h) =0$ for any $w_h \in \cV_h$.
  Note also that $b_z(v, w) = (v, e)_\cH$ for any $v \in
  \cV$. Choosing $v = e$ and applying Lemma~\ref{bzBound},
\revv{
  \begin{align*}
    \| e \|_\cH^2 
    & = (e, e)_\cH = b_z(e, w) = b_z(e, w-w_h)\leq \alpha(z) | e |_{H^1(\om)} | w -   w_h |_{H^1(\om)}~.
  \end{align*}
  Using~\eqref{eq:approxFEM} with $r=\min(s, p)>0$, together
  with~\eqref{eq:w_Hs_bound}, we deduce that
\begin{align*}
\| e \|_\cH^2 \le \alpha(z) C_{\text{app}} h^r |w|_{H^{1+r}(\om)} | e
  |_{H^1(\om)} \leq \alpha(z)\beta(z) C_{\text{app}}\Creg h^r
  \|e\|_\cH | e  |_{H^1(\om)},
\end{align*}}%
\revv{i.e., }
  \begin{equation}
    \label{eq:e-L1H1}
    \| e \|_\cH \le \revv{\alpha(z)\beta(z) C_{\text{app}}\Creg  }h^r \| e \|_\cV.
  \end{equation}
  
  Next, setting $u = R(z) f$,  observe that for any $v_h \in \cV_h$ we have
  \[
  \bigg|
  z \| e \|_\cH^2 - \|e\|_\cV^2 
  \bigg| = |b_z(e,e)| = |b_z(e, u - v_h)|.
  \]
  Hence~\eqref{eq:e-L1H1} \revv{and Lemma~\ref{bzBound} imply}
\revv{
  \[
    \big[1 - \alpha(z)\beta(z) C_{\text{app}}\Creg |z| h^r\big]\,
    \|e\|_\cV^2 \le \alpha(z)  \| e \|_{\cV}
  \inf_{v_h \in \cV_h}\| u - v_h\|_{\cV}.
  \]
  When $h$ is so small that $
  1 - \alpha(z)\beta(z) C_{\text{app}}\Creg |z| h^r
  >1/2,$
  using~\eqref{eq:approxFEM}, we obtain
  $ \| e\|_\cV \le 2\alpha(z) \Capp h^r |u|_{H^{1+r}(\om)}.  $
  Moreover, since $-\Delta u = A u = R(z) f$, using
  Assumption~\ref{asm:reg} and arguing as in~\eqref{eq:w_Hs_bound}, we
  have $\| u \|_{H^{1+r}(\om) } \le \Creg \beta(z)  \| f \|_\cH$. } Thus,
  \begin{equation}
    \label{eq:Rv}
    \| R(z) f - R_h(z) f \|_\cV
    \le \revv{2 \Capp\Creg\alpha(z)\beta(z)}  h^r \| f \|_{\cH}
  \end{equation}
  for any $f \in \cH$.  Now, if $f$ is in $\cV$, then since
  $\| f \|_\cH \le C\| f \|_\cV$, the bound~\eqref{eq:Rv} proves the
  first inequality in \eqref{eq:R-R_h-V}.  Combining
  with~\eqref{eq:e-L1H1}, we obtain the first inequality
  of~\eqref{eq:R-R_h-H} as well.

  To prove the remaining inequalities, we let $f \in E$ and repeat the
  above argument leading to~\eqref{eq:Rv} with $r_E =\min(s_E, p)$ in place of
  $r$.  This proves the second inequality of~\eqref{eq:R-R_h-V}.
  Then using~\eqref{eq:e-L1H1} (where we cannot, in general, 
  replace $r$ by $r_E$), the second inequality
  of~\eqref{eq:R-R_h-H} also follows.
\end{proof}

With the help of the lemma, we can now prove the main result of this
section.

\begin{theorem} 
  \label{thm:total} 
  Suppose Assumption~\ref{asm:rN} (spectral separation) and
  Assumption~\ref{asm:reg} (elliptic regularity) holds, and that $\dim
  E_h^{(0)}=\dim(S E_h^{(0)})=\dim E$.  Then, there
  are positive constants $C$ and $h_0$ such that for all $h<h_0$, the
  subspace iterates $\Eh \ell$ converge (in $\gap_\cV$) to a space $E_h$
  satisfying
  \begin{align}
    \label{eq:gapVEEhFEM}
    \gap_\cV (E, E_h) 
    & \le    C\,  h^{r_E},
    \\
    \label{eq:gapHEEhFEM}
    \gap_\cH (E, E_h) 
    & \le    C\,  h^{r_E+r},
    \\
    \label{eq:LLhFEM}
    \hdist( \Lambda, \Lambda_h) 
    & \le    C\,  h^{2r_E},
  \end{align}
  where $r = \min(s, p)$ and $r_E = \min(s_E, p)$.
\end{theorem}
\begin{proof}
  The proof proceeds by applying the previous theorems after verifying
  their assumptions.  We have already verified
  Assumption~\ref{asm:Vshort} with above set $\cV = H_0^1(\om)$.  In
  view of \eqref{eq:R-R_h-V} of Lemma~\ref{lem:R-Rh:Dir}, since $r>0$,
  Assumption~\ref{asm:Rlim} holds with the same $\cV$.  We may now
  apply Theorem~\ref{thm:iteratesgap}, which yields
  $\gap_\cV(\Eh\ell, E_h) \to 0$. Now the proof
  of~\eqref{eq:gapVEEhFEM} reduces to an application of
  Theorem~\ref{thm:discrete}.
  Next observe that Assumptions~\ref{asm:Vshort} and~\ref{asm:Rlim}
  also hold when $\cV$ is set to $\cH$ (see Example~\ref{eg:case:A}
  and \eqref{eq:R-R_h-H} of Lemma~\ref{lem:R-Rh:Dir}). Applying
  Theorem~\ref{thm:discrete} with this $\cV=\cH$ setting, we obtain
  the estimate~\eqref{eq:gapHEEhFEM}.
Finally,~\eqref{eq:LLhFEM} follows by combining
Corollary~\ref{EigenvalueError_Gaps|A|} with~\eqref{eq:gapVEEhFEM}
and~\eqref{eq:gapHEEhFEM}.
\end{proof}

\section{Numerical Experiments}
\label{sec:Numerics}

We illustrate the convergence results of Theorem~\ref{thm:total} on the model problem
\begin{align*}
-\Delta e=\lambda e\mbox{ in }\Omega\quad,\quad e=0\mbox{ on }\partial\Omega~,
\end{align*}
on three different domains $\Omega\subset\RR^2$.  More specifically,
we consider eigenvalue errors and~\eqref{eq:LLhFEM}.  The experiments
were conducted using~\cite{Gopal17}, which builds a hierarchy of
Python classes representing standard Lagrange finite element
approximations of the filter $S_N$ based on the resolvent
approximations $R_h(z)$, as described in Section~\ref{sec:fem}.   We do not write out the details of the
algorithm because they can be found in our public code~\cite{Gopal17}
or in many previous papers (see e.g., \cite[Algorithm~1.1]{Saad16} and
\cite{GuttePolizTang15}).  As in these references, we perform the
implicit orthogonalization through a small Rayleigh-Ritz eigenproblem
at each iteration.  In general, it is not necessary to perform this
orthogonalization at every step, but in the experiments reported
below, we do so.  For all experiments, filtered subspace iteration is
applied using the Butterworth filter~\eqref{CircleQuad} with $N=8$.
The symmetry (about the real axis) of our filter weights and nodes are
exploited so that only $N/2$ boundary value problems (rather than $N$)
need to be solved for each right hand side per iteration.  The subspace iterations are started
with a random subspace of dimension at least as large as the dimension
of $E$, and the algorithm truncates basis vectors that generate Ritz
values that are deemed too far outside the search interval; in all
cases, we choose this initial subspace dimension to be six.
We stop
the iterations when successive Ritz values differ by less than
$10^{-9}$.  Though changing $N$ does, in some cases, change the number
of subspace iterations used to achieve a prescribed error tolerance
(e.g. three iterations for $N=8$ versus two iterations for $N=16$ for
the Dumbbell problem with $p=3$ and $h=2^{-6}$),
it had no effect on the discretization errors reported here, so we do
not discuss this parameter further.
The
finite element discretizations are implemented using a Python
interface into the C++ finite element library
NGSolve~\cite{Schob17}.  Two parameters govern the discretization: $h$
is the maximum edge length in a quasi-uniform triangulation of
$\Omega$, and $p$ is the polynomial degree in each element.

\subsection{Unit Square}
For the unit square $\Omega=(0,1)\times(0,1)$, the eigenvalues and
eigenvectors may be doubly-indexed by
\begin{align*}
\lambda_{m,n}=(m^2+n^2)\pi^2\quad,\quad e_{m,n}=\sin(m\pi x)\sin(n\pi
  y)\quad,\quad m,n\in\NN~.
\end{align*}
For any subset of the spectrum, the corresponding eigenspaces are
analytic ($s_E=\infty$), and the convexity of $\Omega$ ensures that
$s=1$.  Therefore, Theorem~\ref{thm:total} indicates that the
eigenvalue convergence should behave like $\cO(h^{2p})$.  This is
precisely what is observed in Figure~\ref{SquareEigenvalue} both at
the low end of the spectrum, $\Lambda=\{2\pi^2,5\pi^2\}$, and higher in
the spectrum, $\Lambda=\{128\pi^2,130\pi^2\}$.  We note that both
$2\pi^2$ and $128\pi^2$ are simple eigenvalues, $5\pi^2$ is a double
eigenvalue, and $130\pi^2$ is a quadruple eigenvalue.

For the first set of experiments, the search interval $(0,60)$ was
chosen, so $y=\gamma=30$.  In Figure~\ref{SquareEigenvalueLow}, the
eigenvalue error, $\mathrm{dist}(\Lambda,\Lambda_h)$, is given with respect to
$h$ for (fixed) $p=1,2,3$ and decreasing
$h=2^{-3},2^{-4},\ldots,2^{-7}$. 
For the second set of experiments, the search interval was
$(1260,1290)$, so $y=1275$ and $\gamma=15$.  In order to provide
convergence graphs within the same plot for these more highly
oscillatory eigenvectors, we use $h=2^{-5},2^{-6},2^{-7}$ for $p=2$,
and $h=2^{-4},2^{-5},2^{-6},2^{-7}$ for $p=3$, see
Figure~\ref{SquareEigenvalueHigh}.  For coarser spaces, the
approximations of $130\pi^2$ were far enough outside the search
interval to be rejected by the algorithm, and an approximation for
only $128\pi^2$ was obtained.  A plot of the computed basis for the 
five-dimensional eigenspace corresponding to
$\Lambda=\{128\pi^2,130\pi^2\}$
is given in Figure~\ref{SquareEigenmodes}.  If we label the computed
eigenvalues $\lambda_{1}^h<\lambda_{2}^h<\ldots<\lambda_{5}^h$, with
corresponding eigenvectors $e_{j}^h$, $1\leq j\leq 5$, contour plots of these
eigenvectors are given, from left to right, in
Figure~\ref{SquareEigenmodes}.  One sees that $\mathrm{span}\{e_1^h\}$
approximates $\mathrm{span}\{e_{8,8}\}$, and it appears that
$\mathrm{span}\{e_2^h,e_5^h\}$ approximates
$\mathrm{span}\{e_{3,11},e_{11,3}\}$, and that 
$\mathrm{span}\{e_3^h,e_4^h\}$ approximates
$\mathrm{span}\{e_{7,9},e_{9,7}\}$.

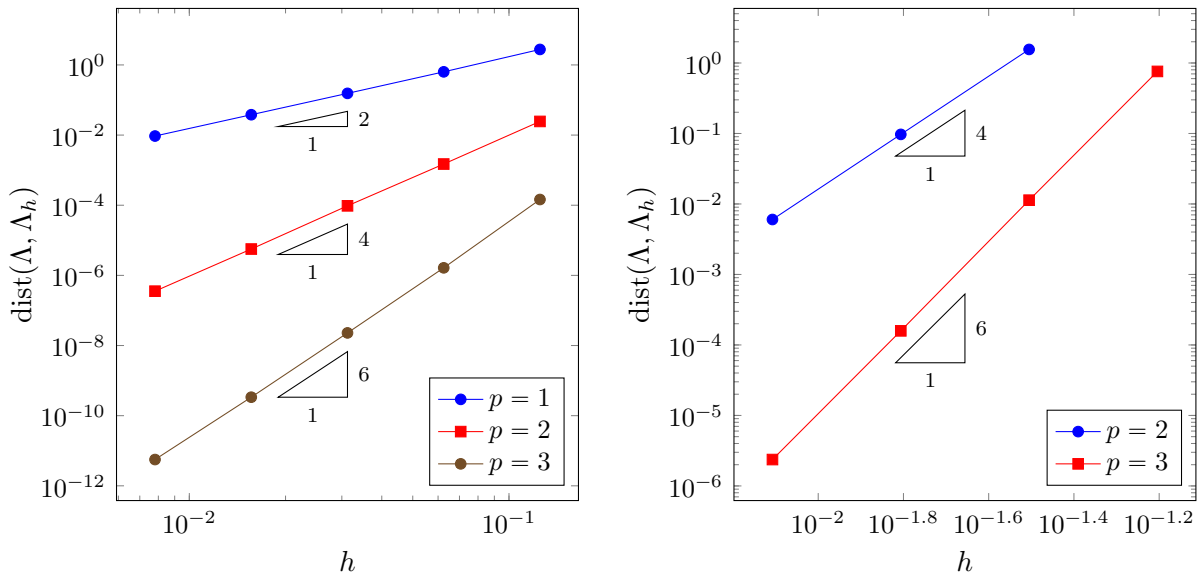
\begin{figure}
\begin{subfigure}[t]{0.49\textwidth} 
     \begin{center}  
      \begin{tikzpicture}
        \begin{loglogaxis}[
          footnotesize,
          width=0.95\textwidth,
          height=\textwidth,
          xlabel=$h$,
          ylabel={$\hdist( {\Lambda}, {\Lambda_h})$},
          legend pos = south east,
          max space between ticks=30pt,
          ]


          \addplot coordinates {
            (1.2500e-01,2.7643e+00)
            (6.2500e-02,6.3527e-01)
            (3.1250e-02,1.5458e-01)
            (1.5625e-02,3.8056e-02)
            (7.8125e-03,9.3877e-03)
           };

          \addplot coordinates {
            (1.2500e-01,2.4497e-02)
            (6.2500e-02,1.4941e-03)
            (3.1250e-02,9.6085e-05)
            (1.5625e-02,5.6346e-06)
            (7.8125e-03,3.5455e-07)
          };

          \addplot coordinates {
            (1.2500e-01,1.4559e-04)
            (6.2500e-02,1.6412e-06)
            (3.1250e-02,2.2897e-08)
            (1.5625e-02,3.3646e-10)
            (7.8125e-03,5.6133e-12)
          };

          \logLogSlopeTriangle{0.5}{0.15}{0.21}{6}{black};
          \logLogSlopeTriangle{0.5}{0.15}{0.5}{4}{black};
          \logLogSlopeTriangle{0.5}{0.15}{0.76}{2}{black};

          \legend{$p=1$, $p=2$, $p=3$}
        \end{loglogaxis}
      \end{tikzpicture}
    \end{center}  
  \caption{\label{SquareEigenvalueLow} Convergence rates for 
    $\Lambda=\{2\pi^2,5\pi^2\}$.}
\end{subfigure}
\begin{subfigure}[t]{0.49\textwidth} 
     \begin{center}  
      \begin{tikzpicture}
        \begin{loglogaxis}[
          footnotesize,
          width=0.95\textwidth,
          height=\textwidth,
          xlabel=$h$,
          ylabel={$\hdist( {\Lambda}, {\Lambda_h})$},
          legend pos = south east,
          max space between ticks=30pt,
          ]


          \addplot coordinates {
            (3.1250e-02,1.5563e+00)
            (1.5625e-02,9.7047e-02)
            (7.8125e-03,6.0208e-03)
          };

          \addplot coordinates {
            (6.2500e-02,7.5914e-01)
            (3.1250e-02,1.1317e-02)
            (1.5625e-02,1.5846e-04)
            (7.8125e-03,2.3673e-06)
          };

          \logLogSlopeTriangle{0.5}{0.15}{0.28}{6}{black};
          \logLogSlopeTriangle{0.5}{0.15}{0.70}{4}{black};
          \legend{$p=2$, $p=3$}
        \end{loglogaxis}
      \end{tikzpicture}
    \end{center}  
  \caption{\label{SquareEigenvalueHigh} Convergence rates for 
    $\Lambda=\{128\pi^2,130\pi^2\}$.}
\end{subfigure}
\caption{\label{SquareEigenvalue} Square: Convergence rates for
  clusters located at the bottom and higher up in the spectrum.}
\end{figure}

\begin{figure}
\begin{center}
%
%
%
\includegraphics[height=1.19in]{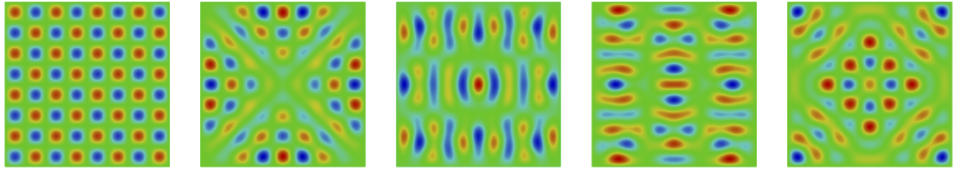}
\end{center}
\caption{\label{SquareEigenmodes} Square: Eigenvectors corresponding
to the eigenvalues $128\pi^2$ (left) and $130\pi^2$.}
\end{figure}

\subsection{L-Shape}

\begin{figure}
\begin{subfigure}[t]{0.32\textwidth} 
     \begin{center}  
      \begin{tikzpicture}
        \begin{loglogaxis}[
          footnotesize,
          width=\textwidth,
          height=2\textwidth,
          ymax=0.25,
          ymin=1.5e-4,
          xlabel=$h$,
          ylabel={$\lambda_h-\lambda$},
          legend pos = north west,
          max space between ticks=30pt,
          ]


          \addplot coordinates {
            (1.2500e-01, 2.1038e-01)
            (6.2500e-02, 6.4855e-02)
            (3.1250e-02, 2.1837e-02)
            (1.5625e-02, 7.8019e-03)
            (7.8125e-03, 2.8998e-03)
           };

          \addplot coordinates {
            (1.2500e-01, 1.8765e-02)
            (6.2500e-02, 6.9814e-03)
            (3.1250e-02, 2.7140e-03)
            (1.5625e-02, 1.0934e-03)
            (7.8125e-03, 4.3037e-04)
          };

          \addplot coordinates {
            (1.2500e-01, 7.2916e-03)
            (6.2500e-02, 2.7382e-03)
            (3.1250e-02, 1.0663e-03)
            (1.5625e-02, 4.2952e-04)
            (7.8125e-03,  1.6916e-04)
          };

          \logLogSlopeTriangle{0.5}{0.15}{0.16}{4/3}{black};
          \logLogSlopeTriangle{0.5}{0.15}{0.28}{4/3}{black};
          \logLogSlopeTriangle{0.5}{0.15}{0.55}{4/3}{black};

          \legend{$p=1$, $p=2$, $p=3$}
        \end{loglogaxis}
      \end{tikzpicture}
    \end{center}  
  \caption{\label{LShapeEigenvalue1} 
    Convergence rates  
    for\\ $\lambda_1\approx 9.6397238$.}
\end{subfigure}
\begin{subfigure}[t]{0.32\textwidth} 
     \begin{center}  
      \begin{tikzpicture}
        \begin{loglogaxis}[
          footnotesize,
          width=\textwidth,
          height=2\textwidth,
          xlabel=$h$,
          legend pos = north west,
          max space between ticks=30pt,
          ]


          \addplot coordinates {
           (1.2500e-01, 2.5227e-01)
            (6.2500e-02, 5.9259e-02)
            (3.1250e-02, 1.4622e-02)
            (1.5625e-02, 3.6044e-03)
            (7.8125e-03, 9.0425e-04)
          };

          \addplot coordinates {
           (1.2500e-01, 1.3977e-03)
            (6.2500e-02, 1.3169e-04)
            (3.1250e-02, 1.5885e-05)
            (1.5625e-02, 2.2339e-06)
            (7.8125e-03, 2.7152e-07)
          };

          \addplot coordinates {
            (1.2500e-01, 7.4197e-05 )
            (6.2500e-02, 1.0461e-05)
            (3.1250e-02, 1.5278e-06)
            (1.5625e-02, 1.8630e-07)
           };

          \logLogSlopeTriangle{0.5}{0.15}{0.24}{3}{black};
          \logLogSlopeTriangle{0.5}{0.15}{0.10}{3}{black};
          \logLogSlopeTriangle{0.5}{0.15}{0.66}{2}{black};
          \legend{$p=1$, $p=2$, $p=3$}
        \end{loglogaxis}
      \end{tikzpicture}
    \end{center}  
  \caption{\label{LshapeEigenvalue2} Convergence rates for \\
    $\lambda_2\approx 15.197252$.}
\end{subfigure}
\begin{subfigure}[t]{0.32\textwidth} 
     \begin{center}  
      \begin{tikzpicture}
        \begin{loglogaxis}[
          footnotesize,
          width=\textwidth,
          height=2\textwidth,
          ymin = 1e-14,
          xlabel=$h$,
          legend pos = south east,
          max space between ticks=30pt,
          ]


         \addplot coordinates {
            (6.2500e-02, 1.0169e-01)
            (3.1250e-02, 2.5019e-02)
            (1.5625e-02, 6.1348e-03)
            (7.8125e-03, 1.5460e-03)
          };

          \addplot coordinates {
           (1.2500e-01, 1.7049e-03)
            (6.2500e-02, 9.9041e-05)
            (3.1250e-02, 6.0110e-06)
            (1.5625e-02, 3.6672e-07)
            (7.8125e-03, 2.3298e-08)
          };
       \addplot coordinates {
         (1.2500e-01, 3.3925e-06)
         (6.2500e-02, 4.1473e-08)
         (3.1250e-02, 6.1214e-10)
         (1.5625e-02, 8.6224e-12)
         };
          \logLogSlopeTriangle{0.5}{0.15}{0.22}{6}{black};
          \logLogSlopeTriangle{0.5}{0.15}{0.54}{4}{black};
          \logLogSlopeTriangle{0.5}{0.15}{0.82}{2}{black};
          \legend{$p=1$, $p=2$, $p=3$}
        \end{loglogaxis}
      \end{tikzpicture}
    \end{center}  
  \caption{\label{LshapeEigenvalue3} Convergence rates for  \\
    $\lambda_3= 2\pi^2$.}
\end{subfigure}
\begin{subfigure}[t]{\textwidth} 
  \centering 
  \includegraphics[height=1.19in]{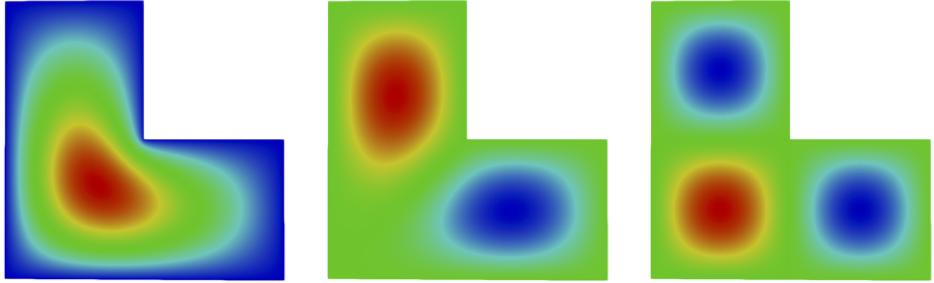}
  \caption{\label{LShapeEigenmodes} L-Shape: Contour plots of the
  computed approximations of the first
  three eigenvectors.}
\end{subfigure}
\caption{\label{LShapeEigenvalue} L-Shape: Convergence rates for 
  $\lambda_1$, $\lambda_2$ and $\lambda_3$.}
\end{figure}
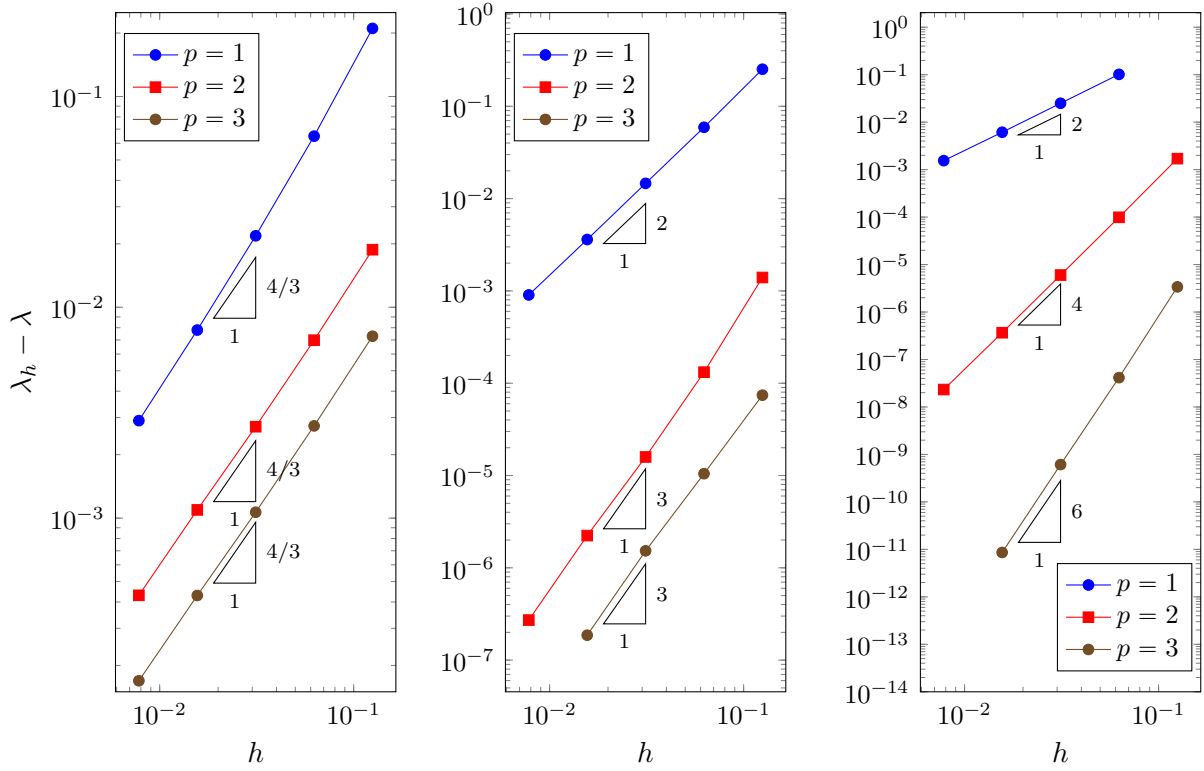
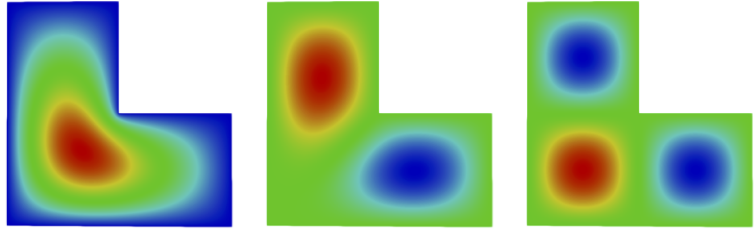

Let $\Omega$ be the $L$-shaped domain that is the concatenation of three
unit squares; see Figure~\ref{LShapeEigenmodes}.  In~\cite{Trefethen2006}, the authors provide very precise
approximations of several eigenvalues for this domain (and other
planar domains).  Their approximations of the first three eigenvalues
(accurate to eight digits) are
\begin{align*}
\lambda_1\approx9.6397238\quad,\quad\lambda_2\approx 15.197252\quad,\quad \lambda_3=2\pi^2\approx 19.739209~,
\end{align*}
and we take the first two of these approximations to be the ``truth'' for purposes of
our convergence studies.  We use the search interval $(0,20)$.

These eigenvalues correspond to eigenvectors having very different
regularities, and the convergence plots in Figure~\ref{LShapeEigenvalue} illustrate that~\eqref{eq:LLhFEM} can
be pessimistic in the sense that it ascribes a single convergence rate
to an entire eigenvalue cluster, and this convergence rate is dictated
by the worst-case regularity of eigenvectors associated with the
cluster.  What we see in practice is that individual eigenvalues
within a cluster converge at rates determined by the regularity of
their corresponding eigenvectors.
Since $\Omega$ has a re-entrant corner with interior angle $3\pi/2$,
we have $r=\min(s,p)=s$ for any $s<2/3$, and the first eigenvector
actually has this regularity.  As such, Theorem~\ref{thm:total}
indicates essentially $\cO(h^{4/3})$ convergence for the cluster.
While this is true for the cluster as a whole, it is only the first
eigenvalue that converges this slowly.  The convergence order for the
second eigenvalue $\cO(h^{\min(2p,3)})$, is consistent with a regularity index $s\leq 3/2$;
and the convergence order for the third eigenvalue, $\cO(h^{2p})$, is precisely what is
expected from an analytic eigenvector.

\subsection{Dumbbell}
Let $\Omega$ be the dumbbell-shaped domain that is a concatenation of two unit-squares
joined by a $1/4\times 1/4$ square ``bridge''; see
Figure~\ref{DumbbellEigenmodes}.  By tiling the dumbbell with
$(1/8)\times(1/8)$ squares, we see that 
$\lambda=128\pi^2\approx 1263.309$ is an eigenvalue, with
corresponding eigenvector $e=\sin(8\pi x)\sin(8\pi y)$.  In order to
determine whether there are other eigenvalues near $128\pi^2$, we
choose the search interval $(1262, 1264)$.  Because of the highly
oscillatory nature of the eigenvector $e$, we employ relatively fine
discretizations, taking $p=3$ and $h=2^{-4}, 2^{-5}, 2^{-6}, 2^{-7}$.
We have determined that there is one other eigenvalue in the search
interval, and it is approximately $1262.41$.  Labeling these
eigenvalues $\lambda_1\leq\lambda_2$, their computed approximations
are given in Table~\ref{DumbbellEigenvalues}.  For the coarsest of
these discretizations, the computed approximation of $128\pi^2$,
$1264.02$, lies slightly outside the search interval, but is accepted
by the algorithm.  Since $\lambda_2=128\pi^2$ is known, we underline the
number of correct digits in our approximations of it.  The error in
this approximation is consistent with $\cO(h^6)$ eigenvalue error,
in agreement with Theorem~\ref{thm:total}.

\begin{table}
\caption{\label{DumbbellEigenvalues} Dumbbell: Computed eigenvalues for
  the interval $(1262,1264)$, $p=3$ and 
  mesh parameters $h=2^{-4}, 2^{-5}, 2^{-6}, 2^{-7}$.}
\begin{tabular}{ccc}
$h$&$\lambda_1$&$\lambda_2$\\\hline
$2^{-4}$&1263.178867&\underline{126}4.020566\\
$2^{-5}$&1262.447629&\underline{1263.3}19956\\
$2^{-6}$&1262.418298&\underline{1263.309}521\\
$2^{-7}$&1262.410062&\underline{1263.30936}6\\
\end{tabular}
\end{table}
\begin{figure}
\begin{center}
\includegraphics[height=1.19in]{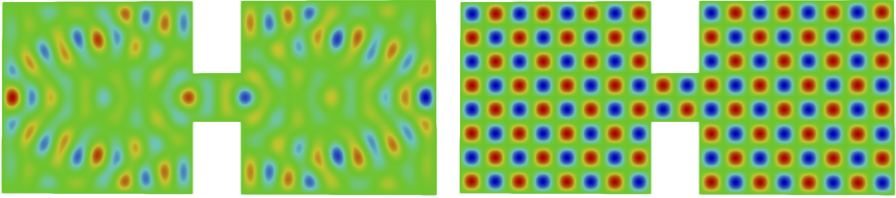}
\end{center}
\caption{\label{DumbbellEigenmodes} Dumbbell: Eigenvectors corresponding
to the eigenvalues $1262.41$ (left) and $128\pi^2$.}
\end{figure}

\def\cprime{$'$}

\end{document}

%% file: definition.tex
\usepackage{amssymb,amsmath,amsthm}

\usepackage[margin=2.5cm,footskip=1.5cm,headsep=1.5cm]{geometry}
\usepackage{tikz,pgfplots}
\usepackage{algorithm}
\usepackage{algorithmic}
\usepackage{subcaption}
\usepackage{cite}
\usepackage{listings}
\usepackage{multirow}
\usepackage{color}
\usepackage[centertags]{amsmath}
\usepackage{mathabx}
\usepackage{enumitem}
\setlist[itemize]{noitemsep, topsep=0pt, leftmargin=0cm}
\usepackage{amscd}

\usepackage{color}

\usepackage{epstopdf}
\usepackage{alltt}
\usepackage{float}
\usepackage{url}

\usepackage{framed}
\newtheorem{theorem}{Theorem}[section]
\newtheorem{lemma}[theorem]{Lemma}
\newtheorem{corollary}[theorem]{Corollary}

\theoremstyle{remark}
\newtheorem{example}[theorem]{Example}
\newtheorem{assumption}[]{Assumption}
\newtheorem{remark}[theorem]{Remark}
\newtheorem{definition}[theorem]{Definition}

 \DeclareMathOperator{\grad}{grad}

\DeclareMathOperator{\gap}{gap}

\newcommand{\Capp}{C_{\mathrm{app}}}
\newcommand{\Creg}{C_{\mathrm{reg}}}

\newcommand{\om}{\varOmega}
\newcommand{\oh}{\varOmega_h}

\newcommand{\veps}{\varepsilon}
\newcommand{\dist}[1]{\mathop{\textstyle{\mathrm{dist}_{{#1}}}}}
\newcommand{\dom}{\mathop{\mathrm{dom}}}

\newcommand{\ran}{\mathop{\mathrm{ran}}}
\newcommand{\CCC}{\mathbb{C}}

\renewcommand{\Re}{{\ensuremath{\mathrm{Re\,}}}}

\newcommand{\R}{\mathbb{R}}
\newcommand{\C}{\mathbb{C}}

\newcommand{\cA}{A}

\newcommand{\cH}{{\mathcal H}}

\newcommand{\cO}{{\mathcal O}}

\newcommand{\cV}{{\mathcal V}}

\newcommand{\NN}{\mathbb{N}}
\newcommand{\RR}{\mathbb{R}}

\newcommand{\ii}{\mathrm{i}}

\newcommand{\logLogSlopeTriangle}[5]
{

    \pgfplotsextra
    {
        \pgfkeysgetvalue{/pgfplots/xmin}{\xmin}
        \pgfkeysgetvalue{/pgfplots/xmax}{\xmax}
        \pgfkeysgetvalue{/pgfplots/ymin}{\ymin}
        \pgfkeysgetvalue{/pgfplots/ymax}{\ymax}

        \pgfmathsetmacro{\xArel}{#1}
        \pgfmathsetmacro{\yArel}{#3}
        \pgfmathsetmacro{\xBrel}{#1-#2}
        \pgfmathsetmacro{\yBrel}{\yArel}
        \pgfmathsetmacro{\xCrel}{\xArel}

        \pgfmathsetmacro{\lnxB}{\xmin*(1-(#1-#2))+\xmax*(#1-#2)} 
        \pgfmathsetmacro{\lnxA}{\xmin*(1-#1)+\xmax*#1} 
        \pgfmathsetmacro{\lnyA}{\ymin*(1-#3)+\ymax*#3} 
        \pgfmathsetmacro{\lnyC}{\lnyA+#4*(\lnxA-\lnxB)}
        \pgfmathsetmacro{\yCrel}{\lnyC-\ymin)/(\ymax-\ymin)} 

        \coordinate (A) at (rel axis cs:\xArel,\yArel);
        \coordinate (B) at (rel axis cs:\xBrel,\yBrel);
        \coordinate (C) at (rel axis cs:\xCrel,\yCrel);

        \draw[#5]   (A)-- node[pos=0.5,anchor=north] {{\tiny{$1$}}}
                    (B)-- 
                    (C)-- node[pos=0.5,anchor=west] {{\tiny{#4}}}
                    cycle;
    }
}

\newcommand{\RRR}{\mathbb{R}}

\newcommand{\qii}[1]{q_i^{\!({{#1}})}}

\newcommand{\Eh}[1]{{E_h^{({#1})}}}
\newcommand{\Eht}[1]{\tilde{E}_h^{({#1})}}

\newcommand{\hdist}{\mathop{\mathrm{dist}}}

\newcommand{\et}{\tilde{e}}

  \newcommand{\revv}[1]{{#1}}

\newcommand{\caF}{c_{a, F}}

\newcommand{\lmax}{\Lambda_h^{\mathrm{max}}}
